\documentclass[12pt,amstex]{amsart}

\usepackage{epsfig}
\usepackage{amsmath}
\usepackage{amssymb}
\usepackage{amscd}
\usepackage{graphicx}
\usepackage[all]{xy}
\usepackage{color}
\usepackage{amsthm, enumitem}
\usepackage[colorlinks=true,citecolor=blue]{hyperref}
\usepackage{stmaryrd}
\usepackage{epsfig}
\usepackage{amsmath,bm}
\usepackage{amssymb}
\usepackage{amssymb,amsthm,amsmath}
\usepackage{mathrsfs}
\usepackage{amscd}
\usepackage{graphicx}
\usepackage{pstricks}
\usepackage{theoremref}

\usepackage{lipsum}

\allowdisplaybreaks

\newtheorem{thm}{Theorem}[section]
\newtheorem{lem}[thm]{Lemma}

\newtheorem{prop}[thm]{Proposition}

\newtheorem{cor}[thm]{Corollary}

\newtheorem{Maintheorem}{Theorem}

\theoremstyle{definition}
\newtheorem{de}[thm]{Definition}

\theoremstyle{remark}
\newtheorem{rem}[thm]{Remark}

\numberwithin{equation}{section}

\def \N {\mathbb{N}}
\def \C {\mathcal C}
\def \Z {\mathbb{Z}}
\def \R {\mathbb{R}}

\def \O {\mathcal{O}}

\def \F {\mathcal F}
\def \G {\mathcal{G}}
\def \B {\mathcal B}
\def \E {\mathbb E}
\def \H {\mathcal H}
\def \K {\mathcal K}

\def \W {\mathcal W}
\def \X {\mathcal{X}}
\def \Z {\mathcal Z}

\def \Y {\mathcal{Y}}

\def \J {\mathcal{J}}
\def \RP {{\bf RP}}

\def \ep {\epsilon}

\def \ol {\overline}

\def \ra {\rightarrow}

\def \RP {\textbf{RP}}

\def \x {\textbf{x}}

\def \0 {\bf 0}

\def \ZZ {\mathbb Z}
\begin{document}
	\title{CF-Nil systems and convergence of two-dimensional ergodic averages}
	
	\author{Kangbo Ouyang, Qinqi Wu}

	\address{Department of Mathematics, University of Science and Technology of China, Hefei, Anhui, 230026, P.R. China}
	
	\email{oy19981231@mail.ustc.edu.cn}

	\address{School of Mathematics, Shanghai university of Finance and Economics, Shanghai, 200433, P.R. China}
	\email{wuqinqi@mail.shufe.edu.cn}


	
	\date{\today}
	\begin{abstract}
		A topological dynamical system $(X,T)$ is called CF-Nil($k$) if it is strictly ergodic and the maximal measurable and maximal topological $k$-step pro-nilfactors coincide as measure preserving systems. Through constructing specific ``CF-Nil'' models,
		we prove that for any ergodic system $(X,\mathcal{X},\mu,T)$,
		any nilsequence $\{\psi(m,n)\}_{m,n\in\mathbb{Z}}$ and any $f_1,\dots,f_d\in L^{\infty}(\mu)$, the averages
		\begin{equation*}
			\dfrac{1}{N^{2}} \sum_{m,n=0}^{N-1} \psi(m,n)\prod_{j=1}^{d}f_{j}(T^{{m+jn}}x)
		\end{equation*}
		converge pointwise as $N$ goes to infinity.  Moreover, we show the $L^2$-convergence of a certain two-dimensional averages for non-commuting transformations without zero entropy condition.
	\end{abstract}

	\maketitle
	\section{Introduction}
Throughout this paper, by a {\it topological dynamical system} (t.d.s. for short) we mean a pair $(X,T)$, where $X$ is a compact metric space and $T$ is a homeomorphism from X onto itself. A {\it measure preserving system} (m.p.s. for short) is a quadruple
$(X,\X,\mu,T)$, where $(X,\X,\mu)$ is a Lebesgue probability space and $T: X \ra X$ is an invertible measure preserving transformation.

\medskip

In recent years, there has been growing interest in studying pro-nilfactors for both m.p.s. and t.d.s.
For a m.p.s. $(X,\X ,\mu, T)$, the maximal measurable pro-nilfactors appear in connection with the characteristic factor of the averages
\begin{equation}\label{average}
	\dfrac{1}{N}\sum_{n=0}^{N-1}f_1(T^n x)\dots f_k(T^{kn}x)
\end{equation}
in $L^2$ norm \cite{HK05,Z}.
In their seminal work, Host and Kra \cite{HK05} introduced ergodic seminorms to construct the characteristic factor for the averages (\ref{average}), which is an inverse limit of nilsystems.
In the theory of t.d.s. maximal topological pro-nilfactors appear in connection with the higher order regionally proximal relations, see \cite{HKM, SY}.

\medskip

Consider a strictly ergodic system $(X, \mu, T)$, which is a minimal t.d.s. with a unique invariant measure $\mu$. For each integer $k \in \mathbb{N}$, we denote the maximal $k$-step measurable pro-nilfactor of $(X, T)$ as $(Z_k(X), \mathcal{Z}_k(X),\mu_k, T)$, and the maximal $k$-step topological pro-nilfactor as $(W_k(X), T)$. It follows directly that $(W_k(X), T)$, equipped with its Borel $\sigma$-algebra $\mathcal{W}_k(X)$, admits a unique invariant measure $\omega_k$. A t.d.s. $(X,T)$ is said to be ``Continuously Factor on a $k$-step pro-Nilsystem which is isomorphic to $(Z_k(X), \mathcal{Z}_k(X), \mu_k, T)$'' ({\it CF-Nil(k)} for short) if it is strictly ergodic and $(W_k(X),\mathcal{W}_k(X), \omega_k, T)$ is isomorphic to $(Z_k(X), \mathcal{Z}_k(X), \mu_k, T)$ as a measure preserving system. 
In \cite{GL}, Gutman and Lian studied such special systems and gave two characterizations.
In this paper we call a t.d.s. $(X,T)$ as a {\it CF-Nil($\infty$) system} if it satisfies the CF-Nil($k$) condition for all $k \in \mathbb{N}$. For the definition of  CF-Nil($\infty$) for $\ZZ^2$-action, see Section \ref{sec2.7}.

\medskip
Let $(X,\X,\mu,T)$ be an ergodic m.p.s. We say that $(\hat{X},\hat{T})$ is a {\it topological model} for $(X,\X,\mu,T)$ if $(\hat{X},\hat{T})$ is a t.d.s. and there exists an invariant probability measure $\hat{\mu}$ on the Borel $\sigma$-algebra $\B(\hat{X})$ such that the systems $(X,\X,\mu,T)$ and $(\hat{X},\B(\hat{X}),\hat{\mu},\hat{T})$ are isomorphic as m.p.s. The well-known Jewett-Krieger Theorem \cite{J,K} states that every ergodic system has a strictly ergodic model. Weiss \cite{W} extended Jewett-Krieger theorem to the relative cases. 
Moreover Weiss \cite{W} showed  that Jewett-Krieger
theorem can be generalized from $\ZZ$-actions to commutative group actions. 

\medskip
Let $(X,T)$ be a t.d.s. and an $x\in X$, and we let $\G_d(T)$ be the group generated by
$$\tau_d(T):=T\times T^2\times\dots\times T^d\text{ and }
T^{(d)}:=T\times\dots\times T(d\text{-times}).$$
In the case $(X,T)$ is minimal, $(x,\dots,x)$ generates the same orbit under the action $\G_d$(T) for all $x\in X$, and we denote this closure by $N_d(X,T)$. For minimal $(X,T)$, Glasner \cite{Glasner22} proved that $(N_d(X,T), \G_d(T))$ is minimal. It is also known that the minimality of $N_d(X)$ implies van der Waerden’s theorem \cite[Theorem 1.56]{vdw}.

\medskip
Recently, several results have been obtained on the limiting behavior of the averages along cubes \cite{Assani,CF,FH,HK05}.
In \cite{HSY19}, by showing that every ergodic system has a topological model $(\hat{X},\hat{T})$ such that $(N_d(\hat{X},\hat{T}),\G_d(\hat{T})$ is strictly ergodic,
Huang, Shao and Ye proved that the averages \begin{equation}\label{HSY}
	\dfrac{1}{N^{2}}\sum_{m,n=0}^{N-1} \prod_{j=1}^d f_{j}(T^{m+nj}x)
\end{equation}
converge pointwise as $N$ goes to infinity. The first result of this article is
\medskip

\begin{Maintheorem}\label{thm-A}
	Let $d\geq 1$ and $(X,\X,\mu,T)$ be an ergodic m.p.s. Then it has a strictly ergodic model $(\hat{X},\hat{T})$ such that
	$(N_d(\hat{X},\hat{T}),\G_d(\hat{T}))$ is CF-Nil($\infty$).
\end{Maintheorem}

\medskip

As an application of Theorem \ref{thm-A}, we extend the pointwise convergence (\ref{HSY}) to encompass cases where sequences are weighted by nilsequences. We state it as:

\medskip
\begin{Maintheorem}\label{thm-B}
	\it  Let $(X,\X,\mu,T)$ be an ergodic m.p.s. and $k,d\geq 1$. Then for any nilsequence $\{\psi(m,n):m,n\in\ZZ\}$, any function $f_{j}\in L^{\infty}(\mu), 1\leq j\leq d$, the averages
	\begin{equation}
		\dfrac{1}{N^{2}}\sum_{m,n=0}^{N-1} \psi(m,n)\prod_{j=1}^{d}f_{j}(T^{m+nj}x)
	\end{equation}
	converge for $\mu$-almost every $x\in X$ as $N$ goes to infinity.
\end{Maintheorem}

\medskip

Let $T,S$ be measure preserving transformations on $(X,\X,\mu)$. It was established by Frantzikinakis and Host \cite{FH} that when $p\in\ZZ[n]$ with $\deg p(n)\geq 2$ and $(X,\mu,T)$ has zero entropy, the multiple ergodic averages $$\lim_{N\ra\infty}\dfrac{1}{N}\sum_{n=1}^{N}T^{n}f\cdot S^{p(n)}g$$ converge in $L^2(\mu)$ for any $f,g\in L^{\infty}(\mu)$. See \cite{HSY-new} for a related result.
The zero entropy assumption on $T$ in the above results is necessary. One can find a counterexample from \cite{FH}.

In this paper, we consider the $L^2$-multiple convergence for two variables. And the assumption of zero entropy of the system can be dispensed with in our scenarios. That is,
\medskip

\begin{Maintheorem}\label{thm-C}
	Let $T,S$ be measure preserving transformations acting on a probability space $(X,\X,\mu)$. Let $d\geq 2$ and $p\in\mathbb{Z}[m,n]$ be a non-linear integer polynomial. Then for any $f_1,\cdots,f_d, g\in L^{\infty}(\mu)$, the limit
	\begin{equation}\label{L2 limit}
		\lim _{N \to \infty} \frac{1}{N^2} \sum_{m,n=1}^{N} S^{p(m,n)}g\cdot \prod_{j=1}^{d} T^{m+j n}f_j 
	\end{equation}
	exists in $L^2(\mu)$.
\end{Maintheorem}

\medskip


Gutman and Lian \cite{GL} showed that CF-Nil($k$) systems and topological Wiener-Witner Theorem are mutually equivalent (see Theorem \ref{GLThmC}). However, we demonstrate the existence of CF-Nil systems where pointwise convergence fails for multiple ergodic averages involving continuous functions.

\medskip
\begin{Maintheorem}\label{thm-D}
	There exists a CF-Nil($\infty$) system $(X,T)$ such that for some $x\in X$ and some $f_1 ,f_2\in \C(X)$, the average  $$\frac{1}{N}\sum_{n=0}^{N-1}f_1(T^n x)f_2(T^{2n} x)$$
	fails to exist as $N$ goes to infinity.
\end{Maintheorem}

\bigskip

\noindent{\bf Organization of the paper.} In Section \ref{s2} we review some fundamental definitions and classical results in nilsystems and Host-Kra structures. In Section \ref{s3}  we prove Theorem \ref{thm-A} and Theorem \ref{thm-B}. We establish Theorem \ref{thm-C} in  Section \ref{s4} which concerns multiple convergence for two-variable polynomial averages. Section \ref{s5} contains the construction of counterexamples which implies Theorem \ref{thm-D}.

\bigskip

\noindent{\bf Acknowledgements.} The authors would like to thank Professors Zhengxing Lian, Song Shao, Leiye Xu and Xiangdong Ye for  useful discussion. This research is supported by National Natural Science Foundation of China (12426201,12371196), Shanghai Sailing Program (24YF2711500) and the Fundamental Research Funds for the Central Universities.

\bigskip

\section{Preliminaries}\label{s2}

In this section, we recall some basic notions and results we need in the following sections. The letters $\R,\ZZ$, $\N$, $\N_+$ stand for the set of real numbers, integers, non-negative integers, and positive integers. 

\subsection{Topological dynamics}

A {\bf  topological dynamical system (t.d.s.)} is a pair
$(X,G)$ consisting of a group $G$ acting on a compact space $X$. For $g\in G$ and $x\in X$ we denote the action by $gx$.
Note that the system $(X,T)$ corresponds to the case of $G$ being $\ZZ$.
In this paper, we always assume that $X$ is a compact metric space with a metric $\rho$.
Denote by $\C(X)$ the set of real-valued continuous functions on $X$. The {\bf  orbit} $\O(x)$ of $x\in X$ is the set $\O(x)=\{gx:g\in G\}$. Its closure is denoted by $\overline{\O}(x)$.
A t.d.s. is {\bf  minimal} if $\overline{\O}(x)=X$ for all $x\in X$.

\medskip

A t.d.s. $(X,G)$ is {\bf  equicontinuous} if for any $\ep> 0$, there is a $\delta>0$ such that whenever $x, y\in X$ with $\rho(x, y)<\delta$, then $\rho(gx,gy)<\ep$ for every $g\in G$. A t.d.s. $(X,G)$ is {\bf  distal} if for any $x\ne y\in X$, $\inf_{g\in G}\rho(gx, gy)>0$. It is well-known that any equicontinuous t.d.s. is distal.

\medskip

A factor map $\pi:X\ra Y$ between the t.d.s. $(X, G)$ and $(Y, G)$ is a continuous and surjective map such that for any $x\in X,g\in G$, $\pi(gx)=g\pi(x)$. Given such a map, $(X,G)$ is called a {\bf  topological factor} of $(Y,G)$. If in addition $\pi$ is injective then it is called a {\bf  topological isomorphism}, $(Y,G)$ and $(X,G)$ are said to be {\bf  isomorphic as t.d.s.} Let $\pi: (X, G)\ra(Y , G)$ be a factor map. Then, $R_\pi= \{(x_1, x_2)\in X\times X:\pi(x_1) = \pi(x_2)\}$ is a closed $G$-invariant equivalence relation and $Y\cong X/R_\pi$. We say that $\pi$ is an  {\bf  almost 1-1 extension} if there exists a dense $G_\delta$ set $X_0\subset X$ such that $\pi^{-1}(\{\pi(x)\}) = \{x\}$ for any $x\in X_0$; is a {\bf  proximal extension} if $R_\pi\subset P(X)$ where  $P(X)=\{(x_1,x_2):\exists g_i\in G \text{ s.t. }\rho(g_i x_1,g_i x_2)\ra 0\}$ is called a proximal relationship. 
Let \(S_{\text{eq}}(X)\) be the smallest closed invariant equivalence relationship \(S\) on \(X\) for which the factor \(X/S\) is an equicontinuous system. We say $X_{eq}:=X/S_{eq}(X)$ is the {\bf maximal equicontinuous factor} of $X$. It is clear that $P(X)\subset S_{eq}(X)$.

\medskip

Denote by $M(X)$ the set of all probability measures on $X$. Let $M_G(X)=\{\mu\in M(G):\mu\circ g^{-1}=\mu,\forall g\in G\}$ be the set of all $G$-invariant Borel probability measures of $X$ and $M^e_G(X)$ be the set of ergodic elements of $M_G(X)$. It is well-known that $M_G^e(X)\ne\varnothing$ when $G$ is an abelian group.
A t.d.s. $(X,G)$ is called {\bf  uniquely ergodic} if there is a uniquely $G$-invariant probability measure on $X$. It is {\bf  strictly ergodic} if it is uniquely ergodic and minimal.
We say $\pi:(X,G)\ra(Y,G)$ is a {\bf topological factor map} if $\pi$ is a continuous and surjective map such that $\pi(gx)=g\pi(x)$ for any $x\in X,g\in G$.

\medskip

\subsection{Measurable systems}
A {\bf  measure preserving  system (m.p.s.)} is a quadruple
$(X,\X,\mu,T)$, where $(X,\X,\mu)$ is a Lebesgue probability space and $T: X \ra X$ is an invertible measure preserving transformation, i.e. $\mu(TA)=\mu(A)$ for all $A\in\X$.

\medskip

For a measurable system $(X,\X,\mu,G)$ we write $\mathcal{I}_{G}$ for the $\sigma$-algebra $\{A\in\X:g^{-1}A=A,\forall\ g\in G\}$ of the invariant sets. 
A m.p.s. $(X,\X,\mu,G)$ is {\bf  ergodic} if all the $G$-invariant sets have measure either $0$ or $1$.
A {\bf  homomorphism} from m.p.s. $(X,\X,\mu,T)$ to m.p.s. $(Y,\Y,\nu,S)$  is a measurable map $\pi:X_0\rightarrow Y_0$, where $X_0$ is a $T$-invariant subset of $X$ and $Y_0$ is a $S$-invariant subset of $Y$, both of full measure, such that $\pi_*\mu=\mu\circ\pi^{-1}=\nu$ and $S\circ\pi(x)=\pi\circ T(x)$ for $x\in X_0$. Given such a homomorphism, $(Y,\Y,\nu,S)$ is called a {\bf  measurable factor} of $(X,\X,\mu,T)$.
If the factor map $\pi:X_0\rightarrow Y_0$ can be chosen to be bijective then $\pi$ is called a {\bf  measurable isomorphism} and $(X,\X,\mu,T)$ and $(Y,\Y,\nu,S)$ are said to be {\bf  isomorphic as m.p.s.} 
A {\bf joining} of two systems $(X, \mu, T)$ and $(Y, \nu, S)$ is a measure $\lambda$ on $X \times Y$, invariant under $T \times S$ and whose projections on $X$ and $Y$ are equal to $\mu$ and $\nu$, respectively.

\medskip

\subsection{Pro-nilsystems and nilsequences}
Let $H$ be a Lie group. Let
$H_1=H$ and $H_k=[H_{k-1},H]$ for $k\geq 2$, where $[H,K]=\{[h, k]:h\in H,k\in K\}$ and $[h,k]=h^{-1}k^{-1}hk$. If there exists some $d\geq 1$ such
that $H_{k+1}= \{e\}$, $H$ is called a {\bf  $k$-step nilpotent} Lie group.
\begin{de}
	Assume that $H$ is  a $k$-step nilpotent Lie group and $\Gamma\subset  H$ is a discrete, cocompact subgroup of $H$. The compact manifold $X=H/\Gamma$ is called a {\bf  $k$-step nilmanifold}. The
	{\bf  Haar measure} $\mu$ of $X$ is the unique Borel probability measure on $X$
	which is invariant under the action $x\mapsto gx$ of $H$ on $X$ by left translations.
	Let $\tau_1,\dots,\tau_s$ be commuting elements of $H$ and $T_i$ be the transformation $x\mapsto \tau_i x$ of $X$, we call $(X,\mu,\langle T_1,\dots, T_s\rangle)$ a {\bf  $k$-step nilsystem}.
\end{de}
An inverse limit
of $k$-step nilsystems is called a {\bf $k$-step pro-nilsystem}.
An inverse limit of nilsystems of finite steps is called an {\bf $\infty$-step pro-nilsystem}.

\begin{thm}\cite{Leib}\label{equi}
	Let $k\geq1 $ be an integer and  $(X=H/\Gamma,\mu,\ZZ^s)$ be $k$-step nilsystem with the Haar measure $\mu$.	Then $(X,\ZZ^s)$ is uniquely ergodic if and only if $(X,\mu,\ZZ^s)$ is ergodic if and only if $(X,\ZZ^s)$ is minimal. 
\end{thm}

We will deal with the space $l^\infty = l^\infty(\mathbb{Z}^s)$ of bounded sequences $\varphi\colon \mathbb{Z}^s \longrightarrow \mathbb{C}$, with the norm $\|\varphi\| = \sup_{n \in \mathbb{Z}^s} |\varphi(n)|$.
\begin{de}
	A bounded sequence $\{\psi(n_1,\dots,n_s):n_i\in\ZZ\}\in l^{\infty}(\ZZ^s)$ is called a {\bf basic} (resp. {\bf smooth}) {\bf $k$-step nilsequence} if there exists a $k$-step nilsystem $(X,G)$ where $G=\langle T_1,\dots, T_s\rangle$ is a $\ZZ^s$-action group, $x_0\in X$ and a continuous (resp. smooth) function $f$ such that
	$$\psi(n_1,\dots,n_s)=f(T_1^{n_1}\dots T_s^{n_s} x_0),\forall(n_1,\dots,n_s)\in\ZZ^s.$$
Denote by $\mathcal{N}_k^o$ the algebra of basic $k$-step nilsequences, and denote by $\mathcal{N}^o$ the algebra $\bigcup_{k \in \mathbb{N}} \mathcal{N}_k^o$ of all basic nilsequences. We will denote the closure of $\mathcal{N}_k^o$, $k \in \mathbb{N}$, in $l^\infty(
\ZZ^s)$ by $\mathcal{N}_k$, and the closure of $\mathcal{N}^o$ by $\mathcal{N}$; the elements of these algebras will be called {\bf $k$-step nilsequences} and, respectively, {\bf nilsequences} (see \cite{HK09,Leib4} for details).
\end{de}

\medskip

\subsection{The Host-Kra measures and Host-Kra seminorms}
Host and Kra introduced {\bf cube measures} defined on $X^{[k]}$ for ergodic $\ZZ$-actions \cite{HK05,HK09,HKbook}.
Bergelson, Tao and Ziegler \cite{BTZ} extended their results to the framework of abelian group actions. We suppose $G$ is generated by commuting transformation $T_1,T_2,\cdots,T_s$ in this section, i.e., $G\cong\ZZ^s$.

\medskip

By a {\bf  cube} of the discrete cube $\{0,1\}^k$ we mean a subcube obtained by fixing some subset of the coordinates.
For $k\in\N$, let $[k]=\{0,1\}^{k}$. Thus $$X^{[k]}=X\times \dots\times X \text{ ($2^k$-times)} \text{ and } \mathcal{T}^{[k]}=\{T_i^{[k]}:i=1,\dots,s \}.$$
A point $\x\in X^{[k]}$ is written as $\x=(x_v:v\in[k])$.
Let $[k]^*=[k]\backslash {\bf 0}$ and define $X_*^{[k]}=X^{[k]^*}$.
\begin{de}\label{cube face}
	Let $(X,\X,\mu,G)$ be an ergodic m.p.s.
	Let $\ol{\alpha}_j=\{v\in[k]:v_j=1\}$ be the {\bf  $j$-th upper face} of $[k]:=\{0,1\}^{k}$.
	For any face $F\subset\{0,1\}^{k}$, for every $g\in G$, define
	\begin{equation*}
		(g^{F})_v=\left\{
		\begin{aligned}
			&g \quad\quad v\in F,\\
			&Id\quad\quad  v\notin F.
		\end{aligned}
		\right.
	\end{equation*}
	Define the {\bf  face group} $\F^k(G)\subset Homeo(X^{[k]})$ to be the group generated by the elements $\{g^{\ol{\alpha}_j}:g\in G,1\leq j\leq k\}$. Define the $k$-th Host-Kra cube group  $\H\K^k(G)$ to be the subgroup of Homeo($X^{[k]}$) generated by
	$\mathcal{T}^{[k]}$ and $\F^{k}(G)$. Let $\J_*^{[k]}$ be the $\sigma$-algebra of sets invariant under $\F^k(G)$ on $X_*^{[k]}$.
\end{de}

\begin{de}
	Let $(X,\X,\mu,G)$ be an ergodic m.p.s. Let $\mu^{[1]}=\mu\times\mu$. For $k\in \N$, let $\mathcal{I}_{\mathcal{T}^{[k]}}$ be the $\mathcal{T}^{[k]}$-invariant $\sigma$-algebra of $(X^{[k]},\X^{[k]},\mu^{[k]})$. Define $\mu^{[k+1]}$ to be the {\bf  relatively independent joining} of two copies of $\mu^{[k]}$ over $\mathcal{I}_{\mathcal{T}^{[k]}}$. That is, for $f_v\in L^{\infty}(\mu),v\in[k+1]$:
	\begin{equation}
		\int_{X^{[k+1]}}\prod_{v\in[k+1]}f_vd\mu^{[k+1]}=
		\int_{X^{[k]}}\E(\prod_{v\in[k]}f_{v0}|\mathcal{I}_{\mathcal{T}^{[k]}})\E(\prod_{v\in[k]}f_{v1}|\mathcal{I}_{\mathcal{T}^{[k]}})d\mu^{[k]}.
	\end{equation}
	
\end{de}
The measures $\{\mu^{[k]}\}_{k\geq 0}$ are referred to as the {\bf Host-Kra measures}. The seminorms $\interleave\cdot\interleave_{k}$ on $L^{\infty}(\mu)$ are defined as 
$$\interleave f\interleave_{k}=(\int\prod_{v\in[k]}C^{|v|}f d\mu^{[k]})^{1/2^k},$$ where $|(v_1,\dots,v_{k})|=\sum_{i=1}^{k}v_i$ and $Cz=\ol{z}$.

\begin{lem}\label{leq|||}\cite[Lemma 3.16]{CS}
	Let $k\geq 1$ and $f_v,v\in[k]$ be $2^k$ bounded functions on $X$. Then
	\begin{equation}
		\left|\int\prod_{v\in[k]}f_v(x_v)d\mu^{[k]}\right|\leq \prod_{v\in[k]}\interleave f_v \interleave_{k}.
	\end{equation}
\end{lem}

\begin{prop}\cite[Lemma 8]{Leib3}\label{leib lemma}
	Let $(X,\X,\mu,T)$ be a m.p.s. and $k \geq 1$. Then for any $f\in L^{\infty}(\mu)$
	$$\lim_{N\ra\infty}\frac{1}{N}\sum_{n=0}^{N-1}\interleave f\cdot T^{an}f\interleave_{k}^{2^{k}}\leq |a|\cdot \interleave f\interleave_{k+1}^{2^{k+1}}$$
\end{prop}

\subsection{The Host-Kra factors}
The following definition is stated in \cite{HKbook} for ergodic $\ZZ$-actions
but easily generalizes to the context of finitely generated abelian actions. Unless stated otherwise, throughout this paper we assume $G=\ZZ^s$ by default.

\begin{de}
	\cite[Subsection 9.1]{HKbook} Let
	$\Z_k(X)$ be the $\sigma$-algebra consisting of measurable sets $B$ such that there exists a $\J_*^{[k+1]}$-measurable set $A\subset X^{[k+1]^*}$ so that up to $\mu^{[k+1]}$-measure zero it holds:
	$$X\times A=B\times X^{[k+1]^*}.$$
	Let $Z_k(X)$ be the measurable factor of $X$ w.r.t. $\Z_k(X)$. Let $\mu_k$ be the projection
	of $\mu$ w.r.t. $X\ra Z_k(X)$. The m.p.s. $(Z_k(X),\mu_k,G)$ is called the  $k$-th {\bf Host-Kra factor} of $(X,\mu,G)$. A system $(X,\mu,G)$ is called a {\bf system of order $k$} if
	$(X,\mu,G) = (Z_k(X),\mu_k,G)$. Note $Z_0(X)$ is the trivial factor and $Z_1(X)$ is the Kronecker factor of $X$.
\end{de}
\begin{rem}
	The Host-Kra factors can also be defined when $G$ is a finitely generated nilpotent action group \cite[Definition 5.9]{CS}.
\end{rem}

\begin{thm}\cite[Lemma 4.3]{HK05}
	Let $k\in\N_+$ and $f\in L^{\infty}(\mu)$. Then $\interleave f\interleave_{k}=0$  if and only if $\E(f|\Z_{k-1}(X))=0$.
\end{thm}
Define $\interleave f\interleave_{\infty}:=\sup_{k\in\N}\interleave f\interleave_{k}$, and it is clear that $\interleave \cdot\interleave_{\infty}$ is also a seminorm. 
Each system of order $\infty$ is an inverse limit of systems of finite orders, and for every system we have $\Z_{\infty}=\bigvee_{k\in\N}\Z_k$.
\begin{thm}\cite[Theorem 10.1]{HK05}\label{structure theorem}
	The $k$-th Host-Kra factor of an ergodic m.p.s. is isomorphic to a minimal $k$-step pro-nilsystem as m.p.s.
\end{thm}
\begin{rem}
	Theorem \ref{structure theorem} is called the structure theorem. Gutman and Lian \cite[Theorem 5.3]{GL19} proved this theorem in the case of finitely generated abelian group actions. Recently, Candela and Szegedy \cite[Theorem 1.2]{CS} extended it to the case of finitely generated nilpotent group actions. It is clear $k$ can be taken as $\infty$.
\end{rem}

Actually, the lemmas and theorems stated in \cite[Section 9, Section 11]{HKbook} can also be generalized easily to finitely generated abelian actions. Then $(X,\mu,G)$ is a system of order $k$ if the seminorm  $\interleave\cdot\interleave_{k+1}$ is a norm, and $Z_k(X)$ is the maximal factor of $X$ that is a system of order $k$.\footnote{It follows from the generalizations of  \cite[Chapter 9, Theorem 15]{HKbook} and \cite[Chapter 9, Theorem 18]{HKbook}.}
Since the $k$-step ergodic pronilsystem is a system of order $k$,\footnote{It follows from the generalizations of  \cite[Chapter 12]{HKbook} and the key point of this fact is that \cite[Chapter 11, Theorem 17]{HKbook} holds when $G$ is a finitely generated abelian group action acting ergodicly on any $k$-step nilmanifold.}  we have that the $k$-th Host-Kra factor of an ergodic system $(X,\mu, G)$ is the maximal factor of $X$ which is a $k$-step pronilsystem as m.p.s.\footnote{This holds since the $k$-th Host-Kra factor is the maximal factor of a system of order $k$ and Theorem \ref{structure theorem} states that any $k$-step pro-nilsystem is a system of order $k$.}
So we also refer to the $k$-th Host-Kra factor as the {\bf maximal $k$-step measurable pro-nilfactor}.

\medskip
The next theorem follows from \cite[Theorem 1.2]{CS} which is an extension of the structure theorem of Host and Kra. 
\begin{thm}
	\label{SC}
	Let $(X,\mu,G)$ be an ergodic system. Then for the $k$-th Host-Kra factor $\Z_k$ and the factor map $\phi_k:X\ra Z_k$, we have
	\begin{enumerate}
		\item $(Z_k,\Z_k,\mu_k,G)$ is the inverse limit of a sequence of $k$-step nilsystems.
		
		\item For every function $f\in L^{\infty}(\mu)$, $\interleave f-\E(f|\Z_k)\circ \phi_k \interleave_{k+1}=0$.
	\end{enumerate}
\end{thm}
\begin{cor}\label{cor}
	Let $(X,\mu,G)$ be an ergodic system and $d\in\N$. Let $f\in L^{\infty}(\mu)$. Then for every $\delta>0$, there exists a $(k-1)$-step ergodic nilsystem $(Y_{\delta},\nu_{\delta},G)$, a (measure theoretic) factor map $p:X\ra Y_{\delta}$, and a continuous function $h$ on $Y$ such that $\interleave f-h\circ p \interleave_{k}<\delta$.
\end{cor}

The following theorem is a generalization of \cite[Theorem 7.1]{HK09}. The proof is similar.
\begin{thm}\label{219}
	Let $(X,\mu,\ZZ^s)$ be a strictly ergodic system and let $\{\psi(n_1,\dots,n_s):n_1,\dots,n_s\in\ZZ^s\}$ be a $k$-step nilsequence. If the factor map $\phi_{k-1}:X\ra Z_{k-1}$ is continuous, then for any bounded continuous function $f\in \C(X)$ and any $x\in X$, the average
	\begin{equation*}
		\dfrac{1}{N^s}\sum_{n_1,\dots,n_s\in[0,N-1]}  \psi(n_1,\dots,n_s)f(T_1^{n_1}\dots T_s^{n_s} x)
	\end{equation*}
	exists. 
\end{thm}

\medskip
\subsection{Maximal topological pro-nilfactors}

\begin{de}
	Let $(X,G)$ be a t.d.s. with $G$ abelian and $k\in \N$. The points $x,y\in X$ are called to be {\bf  regionally proximal of order $k$}, denoted $(x,y)\in \RP^{[k]}(X)$, if there are sequences $f_i\in\F^{k}(G), x_i,y_i\in X, a_*\in X^{[k]}_*$ such that
	$$\lim_{i\ra\infty}(f_i x_i^{[k]},f_i y_i^{[k]})=(x,a_*,y,a_*).$$
\end{de}

Observe that $\RP^{[k+1]}\subset\RP^{[k]}$ for $k\geq 1$. Denote $\RP^{[\infty]}=\cap_{k\in\N}\RP^{[k]}$.
It follows that for any minimal system $(X,G)$ with $G$ abelian and $k\in\N\cup\{\infty\}$, $\RP^{[k]}(X)$ is a closed $G$-invariant equivalence relation \cite{DDMSY,GGY,SY}.
\begin{de}
	Let $k\in\N\cup\{\infty\}$. A t.d.s. $(X,T)$ is called {\bf  a topological system of order $k$} if $\RP^{[k]}(X)=\Delta_X$.
\end{de}
A minimal t.d.s. is a topological system of order $k$ if and only if it is isomorphic to a minimal $k$-step pro-nilsystem as t.d.s.
A minimal system is an $\infty$-step pro-nilsystem if and only if it is an
inverse limit of minimal nilsystems \cite{DDMSY}.
\begin{de}
	Let $(X,G)$ be a minimal t.d.s. with $G$ abelian and $k\geq 1$. Define the {\bf  maximal $k$-step topological pronilfactor } $W_k(X)=X/\RP^{[k]}(X)$. Denote the associated map by $\pi_k:X\ra W_k(X)$.\footnote{For any factor map $\phi:(X,G)\ra(Y,G)$ where $(Y,G)$ is a topological system of order $k$, there exists a unique $\varphi: (W_k(X),G)\ra (Y,G)$ such that $\phi=\varphi\circ\pi_k$.}
\end{de}

\medskip

\subsection{CF-Nil systems.}\label{sec2.7}

\begin{de}
	For $k\geq 0$, we say $(X,T)$ is a {\bf  CF-Nil($k$)} system if $(X,T)$ is strictly ergodic and $(Z_k(X),\Z_k(X),\mu_k,T)$
	is isomorphic to the maximal topological $k$-step pronilfactor  $(W_k(X),\omega_k,T)$
	as m.p.s.
	where $\mu_k$ and $\omega_k$ are the images of the unique invariant measure of $(X,T)$ under the measurable, respectively topological canonical $k$-th projections.
\end{de}
By convention $Z_0(X)=W_0(X)=\{\bullet\}$. Thus every strictly ergodic $(X,T)$ is CF-Nil($0$). Gutman and Lian \cite{GL} provided the correlation between CF-Nil($k$) systems and topological Wiener-Wintner Theorem.

\begin{thm}\label{GLThmC}\cite[Theorem A]{GL}
	Let $(X,T)$ be a minimal t.d.s. Then the following are equivalent for any $k\geq 0$.
	\begin{enumerate}
		\item $(X,T)$ is CF-Nil($k$).
		
		\item For any $k$-step nilsequence $\{\psi(n):n\in\ZZ\}$, any continuous function $f\in \C(X)$ and any $x\in X$,
		\begin{equation*}
			\lim_{N\ra\infty}\dfrac{1}{N}\sum_{n=1}^{N}\psi(n) f(T^{n}x)
		\end{equation*}
		exists.
	\end{enumerate}
\end{thm}

We remark that these definitions can be generalized to ergodic $\ZZ^s$-actions. And (1)$\Rightarrow$(2) in Theorem \ref{GLThmC} holds for $\ZZ^s$-actions (which follows from Theorem \ref{219}).

\begin{de}
	Let $(X,\X,\mu,T)$ be a m.p.s. We say that a t.d.s. $(\hat{X},\hat{T})$ is a {\bf topological model} for $(X,\X,\mu,T)$ with respect to a $\hat{T}$-invariant probability measure $\hat{\mu}$ on the Borel $\sigma$-algebra $\hat{\X}$, if the system $(X,\X,\mu,T)$ is isomorphic to $(\hat{X},\hat{\X},\hat{\mu},\hat{T})$ as m.p.s.
\end{de}
\begin{thm}\cite[Theorem 2]{W}\label{model2}
	If $\pi:(X,\X,\mu,T)\ra (Y,\Y,\nu,S)$ is a factor map with $(X,\X,\mu,T)$ ergodic and $(\hat{Y}, \hat{\Y},\hat{\nu},\hat{S})$ is a uniquely ergodic model for $(Y,\Y,\nu,S)$,
	then there is a uniquely ergodic model $(\hat{X}, \hat{\X},\hat{\mu},\hat{T})$ for $(X,\X,\mu,T)$ and a factor map $\hat{\pi}:\hat{X}\ra \hat{Y}$ which is a model for $\pi:X\ra Y$.
\end{thm}
The following theorem can be found in \cite[Theorem 4.2]{HSY19} and \cite[Theorem 2.34]{GL}.
\begin{thm}\label{model1}
	Let $k\geq 0$. Every ergodic system $(X,\X,\mu,T)$ has a topological model $(\hat{X},\hat{T})$ such that $(\hat{X},\hat{T})$ is a  CF-Nil($k$) system.
\end{thm}

\begin{de}
	We say a t.d.s. $(X,T)$ is a {\bf CF-Nil($\infty$)} system if it is  CF-Nil($k$) for any $k\in\N$.
\end{de}

\begin{prop}\label{everywhere CFNil infty}
	Let $(X,T)$ be a strictly ergodic system with $Z_{\infty}(X)$ and $W_{\infty}(X)$ being measurably isomorphic. Then for any $k\in \mathbb{N}$, any $k$-step basic nilsequence $\{\psi(n)\}_{n \in \mathbb{Z}}$, any continuous function $f \in \C(X)$ and any $x \in X$,
	\begin{equation}\label{infty w-w}
		\lim _{N \rightarrow \infty} \frac{1}{N} \sum_{n=1}^N \psi(n) f\left(T^n x\right)
	\end{equation}
	exists.
\end{prop}
\begin{proof}
	Assume $\left|f\right| \leq 1$ and let $\delta\textgreater 0$. Let $\psi(n)=g(S^n y_0)$ where $(Y,S)$ be a $k$-step nilsystem and $y_0\in Y$, $g\in C(Y)$. 
	The sequence $f(T^n x)$ can be written as $\mathbf{a}^{\prime}+\mathbf{a}^{\prime \prime}$ where $\mathbf{a}^{\prime}$ is a nilsequence and $\mathbf{a}^{\prime \prime}$ is bounded with $\left\|\mathbf{a}^{\prime \prime}\right\|_{U(k)}<\delta$  by \cite[Proof of Proposition 7.1]{HK09}.\footnote{See \cite[Definition 2.6, Definition 5.5]{HK09} for the definitions of norm $\left\|\mathbf\cdot\right\|_{U(k)}$ and dual norm $\interleave \cdot\interleave_k^*$.} 
	
	By density, we can restrict to the case that \( \psi\) is a smooth nilsequence.  Since the product sequence \( \mathbf{a}'\psi \) is a nilsequence, its averages converge. By \cite[Theorem 2.13]{HK09},
	$$
	\limsup \left| \frac{1}{N} \sum_{n=1}^N \left(\psi(n) (f(T^n x) - a'_n)  \right) \right| \leq \delta \interleave \psi\interleave_k^*.
	$$
	It follows that $(\frac{1}{N} \sum_{n=1}^N  \psi(n)f(T^n x)  )$ form a Cauchy sequence.
	We have the limit $$\lim _{N \rightarrow \infty} \frac{1}{N} \sum_{n=1}^N \psi(n) f\left(T^n x\right)$$ exists.
\end{proof}
\begin{prop}
	Let $(X,T)$ be a strictly ergodic system. Then it is CF-Nil($\infty$) if and only if $Z_{\infty}(X)$ and $W_{\infty}(X)$ are measurably isomorphic.
\end{prop}
\begin{proof}
	On one hand, it is clear from the definitions that when $(X,T)$ is CF-Nil($\infty$), then $Z_{\infty}(X)$ and $W_{\infty}(X)$ are measurably isomorphic. 
	On the other hand, if $Z_{\infty}(X)$ and $W_{\infty}(X)$ are measurably isomorphic, by Proposition \ref{everywhere CFNil infty}, for any $\infty$-step nilsequence $\{\psi(n)\}_{n \in \mathbb{Z}}$, any continuous function $f \in \C(X)$ and any $x \in X$, the limit
	$$
	\lim _{N \rightarrow \infty} \frac{1}{N} \sum_{n=1}^N \psi(n) f\left(T^n x\right)
	$$
	exists. Then we have $(X,T)$ is CF-Nil($k$) for any $k\in \mathbb{N}$ by Theorem \ref{GLThmC}.
\end{proof}

\begin{thm}\label{CFNilinfty model}
	Every ergodic system $(X,\X,\mu,T)$ has a topological model $(\hat{X},\hat{T})$ such that $(\hat{X},\hat{T})$ is a CF-Nil($\infty$) system.
\end{thm}
\begin{proof}
	Recall that for every $k\in\N$,  $(Z_k(X),\Z_k(X),\mu_k,T)$ is measurably isomorphic to a strictly ergodic inverse limit of $k$-step nilsystems $(\hat{Z}_k,\hat{T})$ and there is naturally a factor map from $Z_{\infty}(X)$ to $Z_k(X)$. By Theorem \ref{model2}, $(Z_{\infty},\Z_{\infty},\mu_{\infty},T)$ admits a strictly ergodic topological model
	$(\hat{Z}_{\infty},\hat{T})$ such that there exists a topological factor map $(\hat{Z}_{\infty},\hat{T})\ra(\hat{Z}_k,\hat{T})$.	
	By Theorem \ref{model2} again, $(X,\X,\mu,T)$ admits a strictly ergodic topological model
	$(\hat{X},\hat{T})$ such that there exists a topological factor map $(\hat{X},\hat{T})\ra(\hat{Z}_{\infty},\hat{T})$, so the following diagram commutes.
	$$\xymatrix{
		& X \ar[d]\ar[r] & Z_{\infty}(X)
		\ar[d]\ar[r]              & Z_{k}(X) \ar[d]
		\\
		& \hat{X} \ar[r] & \hat{Z}_{\infty}
		\ar[r] & \hat{Z}_{k}
	}
	$$
	
	Thus, $(\hat{X},\hat{T})$ is a strictly ergodic t.d.s. with topological factors on $k$-step pronilsystem $(\hat{Z}_k,\hat{T})$.
	By \cite[Proposition 2.31]{GL} we have  $(\hat{X},\hat{T})$ is CF-Nil($k$) for any $k\in\N$. Thus, $(\hat{X},\hat{T})$ is a CF-Nil($\infty$) system.
\end{proof}

\medskip

\subsection{The induced system $(N_d(X),\G_d)$}
Let $(X,T)$ be a t.d.s. and $d\in \N$. Set $\tau_d:=T\times T^2\times\dots\times T^d,
\sigma_d :=T\times\dots\times T$($d$-times) and note $\G_{d}=\langle\tau_d,\sigma_d
\rangle$.

\medskip

For any $x\in X$, let $N_d(X,x)=\overline{\O}((x,..., x),\G_d)$, the
orbit closure of $(x,\dots,x)$ ($d$-times) under the action of the group $\G_d$. We remark that
if $(X,T)$ is minimal, then all $N_d(X,x)$ coincide, which will be denoted by $N_d(X)$. It was
shown by Glasner \cite{Glasner22} that if $(X,T)$ is minimal, then $(N_d(X),\G_d)$ is minimal.

\begin{de}
	Let $(X,T)$ be a t.d.s with $\mu\in M_T(X)$. For $d\geq1$ and let $\mu^{(d)}$ be the measure on $X^d$ defined by $$\int_{X^d}\bigotimes_{j=1}^{d}f_jd\mu^{(d)}=\lim_{N\rightarrow+\infty}\dfrac{1}{N}\sum_{n=0}^{N-1}\int_{X}\prod_{j=1}^{d}f_j(T^{jn}x)d\mu(x)$$
	for $f_i\in L^{\infty}(X,\mu),1\leq j\leq d$, where the limit exists by \cite{HK05}.
\end{de}
We call $\mu^{(d)}$ the {\bf Furstenberg self-joining}.
Clearly, it is invariant under $\sigma_d$ and $\tau_d$.
For a t.d.s. $(X,T)$, $\mu\in M_T(X)$ and $d\in\N$, it is easy to see that
$$\dfrac{1}{N}\sum_{n=0}^{N-1}\tau_d^n\mu_{\Delta}^d\rightarrow\mu^{(d)},\text{weak* in }M(X^d)$$ as $N$ goes to infinity,
where $\mu^{d}_{\Delta}$ is the diagonal measure on $X^d$ defined as follows
$$\int_{X^d}f_1(x_1)\dots f_d(x_d)d \mu_{\Delta}^{d}(x_1,\dots,x_d)=\int_X f_1(x)\dots f_d(x)d\mu(x),$$
where $f_1,\dots,f_d\in \C(X)$.

\medskip

Denote by $a\lesssim_{d} b$ if there exists a constant $C_d$ depending on $d$ such that $a\leq C_d\cdot b.$ 
\begin{prop}\cite[Proposition 5]{Leib}\label{prop leib}
	Let $(X,\X,\mu,T)$ be a system, $d\geq 1$ and $f_1,\dots,f_d\in\L^{\infty}(\mu)$ bounded by 1. Then
	$$\limsup_{N \to \infty} \left\| \frac{1}{N} \sum_{n=0}^{N-1} \prod_{j=1}^d T^{jn}f_j \right\|_{L^2(X)} \lesssim_{d} \min_{1\leq j\leq d} \interleave f_j \interleave_{d+1}.
	$$
\end{prop}

The following theorem was proved  in \cite[Theorem B]{GHSWY20} for the case $k=1$ and for the general case in \cite[Theorem B]{LQ}.
\begin{thm}\label{pronil Nd}
	Let $(X,T)$ be a minimal system and $k\in\N$. Then for every $d\in\N$, the maximal
	$k$-step pro-nilfactor of $(N_d(X),\G_d)$ is $(N_d(W_k),\G_d).$
\end{thm}
\begin{cor}\label{pronil N}
	Let $(X,T)$ be a minimal system. Then for every $d\in\N$, the maximal
	$\infty$-step pro-nilfactor of $(N_d(X),\G_d)$ is $(N_d(W_\infty),\G_d).$
\end{cor}
\begin{proof}
	Consider the diagram
	$$\xymatrix{
		& N_d(X) \ar[d]_{\pi_\infty}\ar[dr]^{\pi_k}
		\\
		&N_d(W_\infty)\ar[r] & N_d(W_k)
	} $$
	By Theorem \ref{pronil Nd} we have $\mathbf{R}_{\pi_\infty}\subset \mathbf{R}_{\pi_k}=\RP^{[k]}(N_d).$ Then $\mathbf{R}_{\pi_\infty}\subset \RP^{[\infty]}(N_d).$ Clearly by definition $\RP^{[\infty]}(N_d)\subset\mathbf{R}_{\pi_\infty} .$ Therefore $\RP^{[\infty]}(N_d)=\mathbf{R}_{\pi_\infty}$ which imples  the maximal
	$\infty$-step pro-nilfactor of $(N_d(X),\G_d)$ is $(N_d(W_\infty),\G_d).$
\end{proof}
The following theorem plays an important role in the next section.

\begin{thm}\cite[Theorem 1.1]{WZ}\label{WZ}
	Let $(X,T)$ be a minimal t.d.s. and $d\geq 2$. Then $(N_d(X),\G_d)$ has zero topological entropy.
\end{thm}

\medskip

\subsection{The Pinsker factor}
Let $G\cong\mathbb{Z}^s$. Then every system $(X,\mu,G)$ admits a largest factor belonging to the class of systems of entropy zero. This factor is called the \textbf{Pinsker factor} of $(X,\mu,G)$ and is denoted by $\Pi(X,\mu,G)$, or $\Pi(G)$.

\medskip

For a $\mathbb{Z}^2$-action generated by commuting transformations $T$ and $S$ on a probability space $(X, \mu)$, the Pinsker $\sigma$-algebra $\Pi(\mathbb{Z}^2)$ admits the following representation:

\begin{equation}\label{pinsker-pre}
	\Pi(\mathbb{Z}^2) = \bigvee_{\beta \in \mathcal{P}_X} \bigcap_{n=0}^\infty \bigvee_{\substack{(i,j) \in \mathbb{N}^2 \\ i+j \geq n}} T^{-i}S^{-j}\beta
\end{equation}
where $\mathcal{P}_X$ is the set of all finite partitions \cite{JPC}. 
The proof of the following proposition is based on fundamental non-disjointness lemmas in \cite{dlr}.

\begin{prop}\label{FH2.1}
	Let $\rho$ be a joining of a system $(X, \mu, \mathbb{Z}^2)$ of zero entropy and a system $(Y, \nu, \mathbb{Z}^2)$. If $g \in L^2(\nu)$ is such that $\mathbb{E}_\nu(g|\Pi(\mathbb{Z}^2))=0$, then for every $f \in L^2(\mu)$ we have $\int f(x) g(y) \, d\rho(x, y)=0$.
\end{prop}

\begin{proof}
	By \cite[Definition 2.3]{dlr}, \cite[Lemma 4.1]{dlr} and \cite[Proposition 4.3]{dlr}, we obtain that 
	$$\mathbb{E}_\rho (g(y)|X) = \mathbb{E}_\rho(\mathbb{E}_\nu (g(y)| \Pi(\mathbb{Z}^2))|X )=0.$$
	Therefore, 
	$$\int f(x) g(y) \, d\rho(x, y) = \int f(x) \mathbb{E}_\rho (g|X)(y) \, d\rho(x, y) = 0.$$
\end{proof}

By the representation of Pinsker algebra (\ref{pinsker-pre}) and the proof of \cite[Lemma 2.2]{FH}, we can obtain the following lemma.

\begin{lem}\label{FH2.2}
	Let $\pi \colon (X, \mu, \mathbb{Z}^2) \to (Y, \nu, \mathbb{Z}^2)$ be a factor map. 
	If $g \in L^2(\nu)$ is measurable with respect to $\Pi(\mathbb{Z}^2)$, then $g \circ \pi$ is measurable with respect to $\Pi(\mathbb{Z}^2)$.
	If $\mathbb{E}_\nu(g | \Pi(\mathbb{Z}^2)) = 0$, then $\mathbb{E}_\mu(g \circ \pi| \Pi(\mathbb{Z}^2)) = 0$.
\end{lem}

The following lemma is a generalization of \cite[Lemma 2.3]{FH} in the form of \(\mathbb{Z}^2\).
\begin{lem}\label{pinsker}
	Let $(X, \mu, \langle T, S\rangle)$ be a measure-preserving system where $T$ and $S$ are commuting transformations and $f \in L^\infty(\mu)$ be real-valued. Suppose that for every $\ell \geq 0$ and all $(m_1, n_1), \ldots, (m_\ell, n_\ell) \in \mathbb{N}^2$, we have
	\[
	\int f \cdot (T^{m_1}S^{n_1} f) \cdots (T^{m_\ell}S^{n_\ell} f)  \, d\mu = 0.
	\]
	Then $\mathbb{E}_\mu(f  |  \Pi(\mathbb{Z}^2)) = 0$.
\end{lem}

\begin{proof}
	Let $\mathcal{F}$ be the $\sigma$-algebra generated by all functions $\{ T^m S^n f : m, n \in \mathbb{Z} \}$. Since $T$ and $S$ commute, this defines a $\mathbb{Z}^2$-action factor system $(Y, \nu, T, S)$. By Lemma \ref{FH2.2} it suffices to prove that $f$ has zero conditional expectation on the Pinsker factor of $Y$. Thus, we may assume that $\mathcal{F} = \mathcal{B}_X$, where $\mathcal{B}_X$ is the Borel $\sigma$-algebra of $X$.
	
	\medskip
	
	For $N \geq 0$, let $\mathcal{F}_N$ be the $\sigma$-algebra generated by $\{ T^m S^n f : m, n \geq 0, \, m + n \leq N \}$, and define the tail $\sigma$-algebra as $\mathcal{F}_{-\infty} = \bigcap_{N=0}^\infty \mathcal{G}_N$, where $\mathcal{G}_N = \sigma\{ T^m S^n f : m, n \geq 0, \, m + n \geq N \}$. For $\mathbb{Z}^2$-actions, the Pinsker $\sigma$-algebra is contained in $\mathcal{F}_{-\infty}$.
	
	\medskip
	
	We now show that $\mathbb{E}_\mu(f  |  \mathcal{F}_{-\infty}) = 0$. It suffices to prove that for each $N \geq 0$, $\mathbb{E}_\mu(f  |  \mathcal{G}_N) = 0$. Let $\mathcal{A}_N$ be the algebra generated by the functions $\{ T^m S^n f : m, n \geq 0, \, m + n \geq N \}$. Since these functions are bounded, $\mathcal{A}_N$ is dense in $L^2(\mathcal{G}_N)$. Any function $g$ in $\mathcal{A}_N$ is a finite linear combination of functions of the form $\prod_{j=1}^\ell T^{m_j}S^{n_j} f$ with $m_j, n_j \geq 0$ and $m_j + n_j \geq N$. By the hypothesis, for any such product, we have $\int f \cdot \prod_{j=1}^\ell T^{m_j}S^{n_j} f  \, d\mu = 0$. By the linearity of the integral, $\int f \cdot g  \, d\mu = 0$ for all $g \in \mathcal{A}_N$. Since $\mathcal{A}_N$ is dense in $L^2(\mathcal{G}_N)$, it follows that $\mathbb{E}_\mu(f  |  \mathcal{G}_N) = 0$ in $L^2$. As this holds for all $N$, the continuity of conditional expectation implies that $\mathbb{E}_\mu(f  |  \mathcal{F}_{-\infty}) = 0$. Finally, it follows from $\Pi(\mathbb{Z}^2) \subseteq \mathcal{F}_{-\infty}$ that we have $\mathbb{E}_\mu(f  |  \Pi(\mathbb{Z}^2)) = 0$.
\end{proof}

\bigskip

\section{Proofs of Theorems A and B}\label{s3}

In this section, we are going to prove Theorems A and B. For $v\in [k]$, let $\psi(v)=\sum_{i=1}^k  2^{v_i}$.

In \cite{Assani}, Assani gave an estimation for 2-step and 3-step cubic averages and Xiao built specific statements of general case in \cite{Xiao}.  Now we present a lemma   for estimating the high-dimensional cubic average. Our ideas and technical methods are similar to \cite[Lemma 5, Lemma 6]{Assani}, with the essence to let one index vanish.

\begin{lem}\label{inq lem}
	Let $k\geq 2$ and for every $v\in[k]^*$,  $\{a^{\psi(v)}_{m,n}\}_{m, n \in \mathbb{Z}}$ be a real sequence bounded by $1$. Then
	\[\begin{aligned}
		&\left| \frac{1}{N^{2k}}\sum_{l,h\in[0,N-1]^k}\prod_{v\in[k]^*} a^{\psi(v)}_{l\cdot v, k\cdot v} \right|^2 
		\\
		\leq	&\frac{4}{N^{2k-4}} \sum_{\substack{h_1,\ldots,h_{k-2},\\
				l_1,\dots,l_{k-2}=0}}^{N-1} \sup_{t,s} \bigg| \sum_{l_k,h_k=0}^{N-1} e^{2\pi i (l_k t+h_k s)}\cdot\prod_{   \substack{v\in[k]^*:\\ v_{k-1}=0,v_k=1}} a^{\psi(v)}_{l_k + \sum_{i=1}^{k-2} v_i l_i, h_k + \sum_{i=1}^{k-2} v_i h_i} \bigg|^2.\end{aligned}
	\]
\end{lem}

	\begin{proof}
		Let $$C(N,a^{\psi(v)}_{m,n}:v\in[k]^*) = \frac{1}{N^{2k}} \sum_{h_1, \ldots, h_k,l_1,\ldots,l_k=0}^{N-1} \prod_{v \in [k]^*} a^{\psi(v)}_{\sum_{i=1}^k v_i l_i, \sum_{i=1}^k v_i h_i}.$$
		By Cauchy-Schwarz inequality, we have
		\begin{align*}
			&|C(N,a^{\psi(v)}_{m,n}:v\in[k]^*)|^2 \\
			= &\left|\frac{1}{N^{2k-2}} \sum_{\substack{h_1, \ldots, h_{k-1},\\l_1,\ldots,l_{k-1}=0}}^{N-1} \prod_{\substack{v \in[k]^*,\\v_k=0}} a^{\psi(v)}_{\sum_{i=1}^k v_i l_i, \sum_{i=1}^k v_i h_i} \left( \frac{1}{N^2} \sum_{l_k,h_k=0}^{N-1} \prod_{\substack{v \in [k]^*,\\ v_k=1}} a^{\psi(v)}_{\sum_{i=1}^k v_i l_i, \sum_{i=1}^k v_i h_i} \right)\right|^2\\
			\leq & \frac{1}{N^{4k-4}}\left[\sum_{\substack{h_1, \ldots, h_{k-1},\\ l_1\ldots,l_{k-1}=0}}^{N-1}(\prod_{\substack{v \in[k]^*,\\v_k=0}} a^{\psi(v)}_{\sum_{i=1}^k v_i l_i, \sum_{i=1}^k v_i h_i})^2\right]\cdot\\
			&\ \ \ \ \ \ \ \ \ \ \ \ \  \ \ \ \ \ \ \ \  \ \ \ \ \ \ \ \ \ \  \ \ \ \ \  \left[\sum_{\substack{h_1, \ldots, h_{k-1},\\ l_1\ldots,l_{k-1}=0}}^{N-1}\left( \frac{1}{N^2} \sum_{h_k,l_k=0}^{N-1} \prod_{\substack{v\in [k]^*\\ v_k=1}} a^{\psi(v)}_{\sum_{i=1}^k v_i l_i, \sum_{i=1}^k v_i h_i} \right)^2\right]\\
			\leq &  \frac{1}{N^{2k-2}}\sum_{\substack{h_1, \ldots, h_{k-1},\\ l_1\ldots,l_{k-1}=0}}^{N-1}\left| \frac{1}{N^2} \sum_{h_k,l_k=0}^{N-1} \prod_{\substack{v\in [k]^*:\\ v_k=1}} a^{\psi(v)}_{\sum_{i=1}^k v_i l_i, \sum_{i=1}^k v_i h_i} \right|^2.	
		\end{align*}
		Thus, 
		\begin{align*}
			&\frac{1}{N^{2k-2}}\sum_{\substack{h_1, \ldots, h_{k-1},\\ l_1\ldots,l_{k-1}=0}}^{N-1}\left| \frac{1}{N^2} \sum_{h_k,l_k=0}^{N-1} \prod_{\substack{v\in [k]^*:\\ v_k=1}} a^{\psi(v)}_{\sum_{i=1}^k v_i l_i, \sum_{i=1}^k v_i h_i} \right|^2 \\
			=& \frac{1}{N^{2k-2}} \sum_{\substack{h_1,\ldots,h_{k-1},\\ l_1,\ldots,l_{k-1}=0}}^{N-1} \Bigg|\int_{\mathbb{T}^2} \left( \frac{1}{N^2} \sum_{h_k,l_k=0}^{N-1} e^{-2\pi i (h_k t+l_k s)} \prod_{   \substack{v\in[k]^*:\\ v_{k-1}=0,v_k=1}} a^{\psi(v)}_{v_k l_k+\sum_{i=1}^{k-2} v_i l_i, v_k h_k + \sum_{i=1}^{k-2} v_i h_i} \right)\cdot\\
			&\ \ \ \ \ \ \ \ \ \  \left(\sum_{n,m=0}^{2(N-1)} e^{2\pi i (ns+mt)} \prod_{\substack{v \in [k]^*,\\v_{k-1}=1, v_k=1}} a^{\psi(v)}_{n+\sum_{i=1}^{k-2} v_i l_i, m + \sum_{i=1}^{k-2} v_i h_i}\right)e^{-2\pi i(h_{k-1}t+l_{k-1}s)}ds dt\Bigg|^2.
		\end{align*}
		Using Parseval equation, we can get
		\begin{align*}
			& |C(N,a^{\psi(v)}_{m,n}:v\in[k]^*)|^2 \\
			\leq&\frac{1}{N^{2k-2}} \sum_{\substack{h_1,\ldots,h_{k-2},\\l_1,\ldots,l_{k-2}=0}}^{N-1} \Bigg|\int_{\mathbb{T}^2} \left( \frac{1}{N^2} \sum_{h_k,l_k=0}^{N-1} e^{-2\pi i (h_k t+l_k s)} \prod_{v:v_{k-1}=0, v_k=1} a^{\psi(v)}_{v_k l_k+\sum_{i=1}^{k-2} v_i l_i, v_k h_k + \sum_{i=1}^{k-2} v_i h_i} \right)\cdot
			\\& \ \ \ \ \ \ \  \ \left(\sum_{n,m=0}^{2(N-1)} e^{2\pi i (ns+mt)} \prod_{\substack{v \in [k]^*,\\v_{k-1}=1, v_k=1}} a^{\psi(v)}_{n+\sum_{i=1}^{k-2} v_i l_i, m + \sum_{i=1}^{k-2} v_i h_i}\right)\Bigg|^2 ds dt\\
			\leq &\frac{(2N-1)^2}{N^{2k-2}} \sum_{\substack{h_1,\ldots,h_{k-2},\\l_1,\ldots,l_{k-2}=0}}^{N-1} \sup_{t,s}\left|  \frac{1}{N^2} \sum_{h_k,l_k=0}^{N-1} e^{-2\pi i (h_k t+l_k s)} \prod_{v:v_{k-1}=0,v_k=1} a^{\psi(v)}_{ l_k+\sum_{i=1}^{k-2} v_i l_i,  h_k + \sum_{i=1}^{k-2} v_i h_i} \right|^2\\\leq &\frac{4}{N^{2k-4}} \sum_{\substack{h_1,\ldots,h_{k-2},\\l_1,\ldots,l_{k-2}=0}}^{N-1} \sup_{t,s}\left|  \frac{1}{N^2} \sum_{h_k,l_k=0}^{N-1} e^{-2\pi i (h_k t+l_k s)} \prod_{v:v_{k-1}=0,v_k=1} a^{\psi(v)}_{ l_k+\sum_{i=1}^{k-2} v_i l_i,  h_k + \sum_{i=1}^{k-2} v_i h_i} \right|^2.
		\end{align*}
	\end{proof}
	
\begin{lem}\label{useful lemma}
	Let $(X,\X,\mu,T)$ be a m.p.s. and $k \geq 1$. Then for any $f\in L^{\infty}(\mu)$ and any integer $1\leq l\leq k$,
	$$\limsup_{N\ra\infty}\frac{1}{N^2}\sum_{m,n=0}^{N-1}\interleave f\cdot T^{m+jn}f\interleave_{k}^{2^{l}}\leq  \interleave f\interleave_{k+1}^{2^{l+1}}.$$
\end{lem}
\begin{proof}Set $$A_h=\{(m,n):m+jn=h, 0\leq m\leq N-1, 0\leq n\leq N-1 \}$$ for $0\leq h\leq (j+1)(N-1) .$ Then $A_h$ has at most  $(\lfloor\frac{N-1}{j}\rfloor+1)$ elements. Therefore,
	\begin{align*}
		&\limsup_{N\ra\infty}\frac{1}{N^2}\sum_{m,n=0}^{N-1}\interleave f\cdot T^{m+jn}f\interleave_{k}^{2^{k}}\\
		\leq&\lim_{N\ra\infty}\frac{1}{N^2}\sum_{h=0}^{(j+1)(N-1)}(\lfloor\frac{N-1}{j}\rfloor+1)\interleave f\cdot T^{h}f\interleave_{k}^{2^{k}}\\
		\leq&  \interleave  f\interleave_{k+1}^{2^{k+1}}.
	\end{align*}
	The last inequality follows from Proposition \ref{leib lemma}. 
	Then by Cauchy–Schwarz inequality we have
	\begin{align*}
		&\limsup_{N\ra\infty}\frac{1}{N^2}\sum_{m,n=0}^{N-1}\interleave f\cdot T^{m+jn}f\interleave_{k}^{2^{l}}\\\leq&\limsup_{N\ra\infty}\left(\frac{1}{N^2}\sum_{m,n=0}^{N-1}\interleave f\cdot T^{m+jn}f\interleave_{k}^{2^{k}}\right)^{\frac{1}{2^{k-l}}}\\\leq& \interleave  f\interleave_{k+1}^{2^{l+1}}.
	\end{align*}
	
\end{proof}

\begin{lem}\cite[Lemma 3.1]{KN}\label{vdw lem}(van der Corput's Lemma)
	Let $\{u_n\}_{n \in \mathbb{Z}}$ be a bounded complex sequence. For any $N, H \geq 1$ with $H \leq N$, we have
	\[
	H^2 \left| \sum_{n=1}^N u_n \right|^2 \leq H(N + H - 1) \sum_{n=1}^N |u_n|^2 + 2(N + H - 1) \sum_{h=1}^{H - 1} (H - h) \text{Re} \sum_{n=1}^N u_n \overline{u}_{n + h}.
	\]
\end{lem}
The above lemma can be extended to the two-dimensional case. If  $\{u_{m,n}\}_{m,n\in\N}$ is a complex sequence with $|u_{m,n}|\leq 1$. Define $u_{m,n}=0$ for all $(m,n)\notin\N^2$. Then for any $N, H \geq 1$ with $H \leq N$ there exists $C>5$ such that
\begin{equation}\label{vdw2dim}
	\begin{aligned}
	 \left|\dfrac{1}{N^2} \sum_{m,n=1}^N u_{m,n} \right|^2 \leq & \dfrac{C}{H}+
		 \sum_{h_1=1}^{H - 1} \sum_{h_2=-H+1}^{H - 1}\dfrac{2(H - h_1)(H-|h_2|)}{H^4} \text{Re} \sum_{m,n=1}^N u_{m,n} \overline{u}_{m+h_1,n + h_2}.
	\end{aligned}
\end{equation}

\begin{prop}\cite{Xiao}\label{xiao}
	Let $(X,\mathcal{X},\mu,T)$ be a system, $d \in \mathbb{N}, f_1, \cdots, f_d $ be real-valued functions bounded by one. Then
	$$
	\int_{X^d} \limsup _{N \rightarrow \infty} \sup _t\left|\frac{1}{N} \sum_{n=1}^N e^{2 \pi i n t} \prod_{j=1}^d f_j\left(T^{j n} x_j\right)\right| d \mu^{(d)}(\boldsymbol{x})  \lesssim_{d} \min _{1 \leq j \leq d}\interleave f_j\interleave_{d+3}.
	$$
\end{prop}

Since $\mu^{(d)}$ can be seen as the Furstenberg self-joining of $\{n,2n,\dots,dn\}$, we have the next Theorem.
\begin{thm}\label{leq d+3}
	Let $(X,\mathcal{X},\mu,T)$ be a system, $d \in \mathbb{N}, f_1, \cdots, f_d $ be real values functions bounded by one. Then
	$$
	\int_{X^d} \limsup _{N \rightarrow \infty} \sup _{t,s}\left|\frac{1}{N^2} \sum_{m,n=0}^{N-1} e^{2 \pi i (n t+ms)} \prod_{j=1}^d f_j\left(T^{m+j n} x_j\right)\right| d \mu^{(d)}(\x) \lesssim_{d} \min _{1 \leq j \leq d}\interleave f_j \interleave_{d+3}.
	$$
\end{thm}
\begin{proof}
	By (\ref{vdw2dim}), there exists a constant $C>0$ such that
	\begin{align*}
		&\int_{X^d} \limsup _{N \rightarrow \infty} \sup _{t,s}\left|\frac{1}{N^2} \sum_{m,n=0}^{N-1} e^{2 \pi i (n t+ms)} \prod_{j=1}^d f_j\left(T^{m+j n} x_j\right)\right|^2 d \mu^{(d)}(\boldsymbol{x})\\
		\leq &\frac{C}{H}+\frac{C}{H^2}\sum_{h_1=1}^{H - 1} \sum_{h_2=-H+1}^{H - 1}\int_{X^d} \limsup _{N \rightarrow \infty} \left|\frac{1}{N^2} \sum_{m,n=0}^{N-1}  \prod_{j=1}^d T^{m+jn}(f_j\cdot T^{h_1+jh_2}f_j)(x_j)\right| d \mu^{(d)}(\boldsymbol{x})\\
		\leq &\frac{C}{H}+\frac{C}{H^2}\sum_{h_1=1}^{H - 1} \sum_{h_2=-H+1}^{H - 1}\left(\int_{X^d} \limsup _{N \rightarrow \infty} \left|\frac{1}{N^2} \sum_{m,n=0}^{N-1}  \prod_{j=1}^d T^{m+jn}(f_j\cdot T^{h_1+jh_2}f_j)(x_j)\right|^2 d \mu^{(d)}(\boldsymbol{x})\right)^{\frac{1}{2}}\\
	\end{align*}
	for any $H\in\N$. 
	
	By Birkhoff’s pointwise ergodic theorem and (\ref{vdw2dim}) again, we have
	\begin{align*}
		&\int_{X^d} \limsup _{N \rightarrow \infty} \sup _{t,s}\left|\frac{1}{N^2} \sum_{m,n=0}^{N-1} e^{2 \pi i (n t+ms)} \prod_{j=1}^d f_j\left(T^{m+j n} x_j\right)\right|^2 d \mu^{(d)}(\boldsymbol{x})\\
		\leq&\frac{C}{H}+\frac{C}{H^2}\sum_{h_1=1}^{H - 1} \sum_{h_2=-H+1}^{H - 1}\left( \limsup _{N \rightarrow \infty}\int_{X^d} \left|\frac{1}{N^2} \sum_{m,n=0}^{N-1}  \prod_{j=1}^d T^{m+jn}(f_j\cdot T^{h_1+jh_2}f_j)(x_j)\right|^2 d \mu^{(d)}(\boldsymbol{x})\right)^{\frac{1}{2}}\\
		=&\frac{C}{H}+\frac{C}{H^2}\sum_{h_1=1}^{H - 1} \sum_{h_2=-H+1}^{H - 1}\left( \limsup _{N \rightarrow \infty}\int_{X^d} \left|\frac{1}{N} \sum_{n=0}^{N-1}  \prod_{j=1}^d T^{jn}(f_j\cdot T^{h_1+jh_2}f_j)(x_j)\right|^2 d \mu^{(d)}(\boldsymbol{x})\right)^{\frac{1}{2}}\\
		\leq & \frac{C}{H}+\frac{C}{H^2}\sum_{h_1=1}^{H - 1} \sum_{h_2=-H+1}^{H - 1}  \left( \frac{C}{K^2}+\frac{C}{K^2}\limsup_{N \to \infty} \int_{X^d} \left[ \sum_{k_1,k_2=1}^{K-1} \frac{(K-k_1)(K-k_2)}{K^2} \right. \right. \\
		&\left. \left. \frac{1}{N} \sum_{n=0}^{N-1} \prod_{j=1}^d T^{jn}(f_j\cdot T^{h_1+jh_2}f_j)\cdot T^{k_1+jk_2}\left(T^{jn}(f_j\cdot T^{h_1+jh_2}f_j)\right)(x_j) d\mu^{(d)}(\boldsymbol{x}) \right] \right)^{\frac{1}{2}}\\
		\leq & \frac{C}{H}+\frac{C}{H^2}\sum_{h_1=1}^{H - 1} \sum_{h_2=-H+1}^{H - 1}  \left( \frac{C}{K}+\frac{C}{K^2}\sum_{k_1=1}^{K - 1} \sum_{k_2=-K+1}^{K - 1} \frac{(K-k_1)(K-k_2)}{K^2} \limsup_{N \to \infty} \frac{1}{N}\cdot   \right. \\
		&\left.  \sum_{m=1}^\infty\limsup_{m\ra\infty}\left\| \frac{1}{M}\sum_{m=0}^{M-1} \prod_{j=1}^d T^{j(n+m)}(f_j\cdot T^{h_1+jh_2}f_j)\cdot T^{k_1+jk_2}\left(f_j\cdot T^{h_1+jh_2}f_j\right) \right\|_2\right)^{\frac{1}{2}}
	\end{align*}
	for any $K\in\N$. By Proposition \ref{prop leib} we have
	\begin{align*}
		&\int_{X^d} \limsup _{N \rightarrow \infty} \sup _{t,s}\left|\frac{1}{N^2} \sum_{m,n=0}^{N-1} e^{2 \pi i (n t+ms)} \prod_{j=1}^d f_j\left(T^{m+j n} x_j\right)\right|^2 d \mu^{(d)}(\boldsymbol{x})\\
		\lesssim_{d} & \frac{1}{H}+\frac{1}{H^2}\sum_{h_1=1}^{H - 1} \sum_{h_2=-H+1}^{H - 1}  \left( \frac{1}{K}+\frac{1}{K^2}\sum_{k_1=1}^{K - 1} \sum_{k_2=-K+1}^{K - 1} \min_{1\leq j\leq d}\interleave f_j\cdot T^{h_1+jh_2}f_j\cdot T^{k_1+jk_2}(f_j\cdot T^{h_1+jh_2}f_j) \interleave_{d+1}\right)^{\frac{1}{2}}.
	\end{align*}
	By Lemma \ref{useful lemma} and taking limsup of $H$ and $K$, we have
	\begin{align*}
		&\int_{X^d} \limsup _{N \rightarrow \infty} \sup _{t,s}\left|\frac{1}{N^2} \sum_{m,n=0}^{N-1} e^{2 \pi i (n t+ms)} \prod_{j=1}^d f_j\left(T^{m+j n} x_j\right)\right|^2 d \mu^{(d)}(\x) \\
		\lesssim_{d}& \min_{1\leq j\leq d} \limsup_{H\to\infty} \frac{1}{{H}^2} \sum_{h_1=1}^{H - 1} \sum_{h_2=-H+1}^{H - 1} \left( \limsup_{K\to\infty} \frac{1}{K^2} \sum_{k_1=1}^{K - 1} \sum_{k_2=-K+1}^{K - 1} \interleave f_j\cdot T^{h_1+jh_2}f_j\cdot T^{k_1+jk_2}(f_j\cdot T^{h_1+jh_2}f_j) \interleave_{d+1} \right)^{\frac{1}{2}} \\
		\lesssim_{d}& \min_{1\leq j\leq d} \limsup_{H\to\infty} \frac{1}{{H}^2} \sum_{h_1=1}^{H - 1} \sum_{h_2=-H+1}^{H - 1}\interleave f_j \cdot T^{h_1+jh_2} f_j \interleave_{d+2} \\
		\lesssim_{d}& \min_{1\leq j\leq d} \interleave f_j\interleave_{d+3}^2.
	\end{align*}
	
	Therefore, 
	$$
	\int_{X^d} \limsup _{N \rightarrow \infty} \sup _{t,s}\left|\frac{1}{N^2} \sum_{m,n=0}^{N-1} e^{2 \pi i (n t+ms)} \prod_{j=1}^d f_j\left(T^{m+j n} x_j\right)\right| d \mu^{(d)}(\x) \lesssim_{d} \min _{1 \leq j \leq d}\interleave f_j \interleave_{d+3}.
	$$
\end{proof}

\begin{thm}\label{main dom}
	Let  $(X, \mathcal{X}, \mu, T)$ be a system, $d,k \in \mathbb{N}, f_1, \cdots, f_d \in L^{\infty}(\mu)$. Then 
	$$\lim _{H \rightarrow \infty} \frac{1}{H^{2k}} \sum_{h,l \in[0,H-1]^k} \int_{X^d}\bigotimes_{j=1}^{d}\left(\prod_{v \in [k]} T^{j(h \cdot v)+l \cdot v} f_j\right) d \mu^{(d)}(\mathbf{x})\lesssim_d\min _{1 \leq j \leq d}\interleave f_j\interleave_{d+k+1}.$$
\end{thm}
\begin{proof}
	Suppose that $\|f_j\|_\infty\leq 1$. We have 
	\begin{align*}
		&\lim _{H \rightarrow \infty} \frac{1}{H^{2k}} \sum_{h,l \in[H]^k} \int_{X^d}\bigotimes_{j=1}^{d}\left(\prod_{v \in [k]} T^{j(h \cdot v)+l \cdot v} f_j\right) d \mu^{(d)}(\mathbf{x})\\
		\leq&\left(\int_{X^d} \limsup _{H \rightarrow \infty}\left|\frac{1}{H^{2k}} \sum_{h,l \in[H]^k}\bigotimes_{j=1}^{d}\left(\prod_{v \in [k]} T^{j(h \cdot v)+l \cdot v} f_j\right)\right|^{2^{k-2}} d \mu^{(d)}(\mathbf{x})\right)^{\frac{1}{2^{k-2}}} \\=& \left(\int_{X^d} \limsup _{H \rightarrow \infty}\left| \frac{1}{H^{2k}} \sum_{h,l \in[H]^k} \prod_{v \in [k]} G_{l \cdot v, h \cdot v}\right|^{2^{k-2}} d \mu^{(d)}(\mathbf{x})\right)^{\frac{1}{2^{k-2}}}
	\end{align*}
	where $G_{m,n}=T^{m+n}f_1 \otimes \cdots \otimes T^{m+dn}f_d:=\prod_{j=1}^d T^{m+jn}f_j$. 
	
	\medskip
	\noindent{\bf Case 1:} For $k>2$, by Lemma \ref{inq lem} we have
	\begin{align*}
		&\quad H^2  \left| \frac{1}{H^{2k}} \sum_{h,l \in[H]^k}  \left(\prod_{v \in [k]} T^{h \cdot v+l \cdot v} f_j\right) \otimes \cdots \otimes \left(\prod_{v \in [k]} T^{d(h \cdot v)+l \cdot v} f_d\right)\right|^{2^{k-2}} \\
		& \leq 4^k \cdot\sum_{h_1,l_1=0}^{H-1}\left(\frac { 1 } { H^2 } \sum _ { h_{2},l_2 = 0 } ^ { H - 1 } \left(\cdots \left(\left.\frac{1}{H^2} \sum_{h_{k-2},l_{k-2}=0}^{H-1} \sup _{t,s} \right\rvert\, \frac{1}{H^2} \sum_{h_k,l_k=0}^{H-1} e^{2 \pi i (h_k t+l_k s)} .\right.\right.\right.\\
		&\ \ \ \ \ \ \ \ \ \ \ \ \  \left.\left.\left.\left.\prod_{v:v_{k-1}=0,v_k=1} G_{l_k v_k+\sum_{i=1}^{k-2} l_i v_i, h_k v_k+\sum_{i=1}^{k-2} h_i v_i}\right|^2\right)^2 \cdots\right)^2\right)^2 \\
		&\lesssim_{k} \sum_{h_1,l_1=0}^{H-1}\left(\frac { 1 } { H^2 } \sum _ { h_{2},l_2 = 0 } ^ { H - 1 } \left(\cdots \left(\left.\frac{1}{H^2} \sum_{h_{k-2},l_{k-2}=0}^{H-1} \sup _{t,s} \right\rvert\, \frac{1}{H^2} \sum_{h_k,l_k=0}^{H-1} e^{2 \pi i (h_k t+l_k s)} .\right.\right.\right.\\
		&\ \ \ \ \ \ \ \ \ \ \ \ \  \left.\left.\left.\left.\prod_{v:v_{k-1}=0,v_k=1} G_{l_k v_k+\sum_{i=1}^{k-2} l_i v_i, h_k v_k+\sum_{i=1}^{k-2} h_i v_i}\right|\right) \cdots\right)\right).
	\end{align*}
	
	For fixed $(h_1, \cdots, h_{k-2}),(l_1, \cdots, l_{k-2}) \in[H]^{k-2}$, from Theorem \ref{leq d+3} where $f_j$ replaced by  $\prod_{\eta\in[k-2]}T^{(\eta00)\cdot (jh+l)}f_j$, we have 
	
	\begin{align*}	&\int_{X^d} \limsup _{H \rightarrow \infty} \sup _{t,s} \left| \frac{1}{H^2} \sum_{h_{k},l_{k}\in[0,H-1]} e^{2 \pi i (h_{k}\cdot t+l_{k}\cdot s)}\prod_{j=1}^d\bigg( \prod_{\eta\in[k-2]}T^{\eta01\cdot l+j(\eta01\cdot h)} f_j\bigg)\right| d\mu^{(d)}(\x) \\& \lesssim_{k, d}\min_{1\leq j\leq d}\interleave \prod_{v\in[k-2]}T^{\sum_{i=1}^{k-2} l_i v_i+j(\sum_{i=1}^{k-2} h_i v_i)}f_j\interleave_{d+3}.
	\end{align*}
	
	Thus, by Proposition \ref{leib lemma}, Lemma \ref{useful lemma} and $\|f_j\|_\infty\leq 1$ we have that
	\begin{align*}	&\left(\lim _{H \rightarrow \infty} \frac{1}{H^{2k}} \sum_{h,l \in[H]^k} \int_{X^d}\bigotimes_{j=1}^{d}\left(\prod_{v \in [k]} T^{j(h \cdot v)+l \cdot v} f_j\right) d \mu^{(d)}(\mathbf{x})\right)^{2^{k-2}}\\
		\lesssim&_{d}\bigg[\limsup_{H_1}\dfrac{1}{H_1^2}\sum_{h_1,l_1}\bigg(\dots\limsup_{H_{k-2}}\dfrac{1}{H_{k-2}^2}\sum_{h_{k-2},l_{k-2}}\bigg(\interleave \prod_{v\in[k-2]}T^{\sum_{i=1}^{k-2} (l_i+jh_i) v_i}f_j\interleave_{d+3}\bigg)\dots\bigg)\bigg]\\
		\lesssim&_{d}\limsup_{H_1}\dfrac{1}{H_1^2}\sum_{h_1,l_1}\dots\limsup_{H_{k-2}}\dfrac{1}{H_{k-2}^2}\sum_{h_{k-2},l_{k-2}}\interleave (\prod_{v\in[k-3]}T^{\sum_{i=1}^{k-3} (l_i+jh_i) v_i}f_j)\cdot\\
		&\ \ \ \ \ \ \ \ \ \ \ \ \ \ \ \ \ \ \ \ \ \ \ \ \ \ \ \ \ \ \ \ \ \ \ \ \ \ \ \ \ \ \ \ \ \ \ \ \ \ \ \ \ T^{l_{k-2}+jh_{k-2}}(\prod_{v\in[k-3]}T^{\sum_{i=1}^{k-3} (l_i+jh_i) v_i}f_j)\interleave_{d+3}\\
		\lesssim&_{d,k-2}\limsup_{H_1}\dfrac{1}{H_1^2}\sum_{h_1,l_1}\dots\limsup_{H_{k-3}}\dfrac{1}{H_{k-3}^2}\sum_{h_{k-3},l_{k-3}}\interleave \prod_{v\in[k-3]}T^{\sum_{i=1}^{k-3} (l_i+jh_i) v_i}f_j\interleave^2_{d+3+1}
		\\\lesssim&\dots
		\lesssim_{d,k}\interleave f_j\interleave^{2^{d-2}}_{d+3+(k-2)}=\interleave f_j\interleave^{2^{d-2}}_{d+k+1}.
	\end{align*}	
	
	\medskip
	\noindent{\bf Case 2:} For $k=2$, we have 
	\begin{align*}
		&\ \ \ \lim _{H \rightarrow \infty} \frac{1}{H^{4}} \sum_{h,l \in[H]^2} \int_{X^d}\bigotimes_{j=1}^{d}\left(\prod_{v \in [2]} T^{j(h \cdot v)+l \cdot v} f_j\right) d \mu^{(d)}(\mathbf{x})\\
		&\leq  \left(\int_{X^d} \limsup _{H \rightarrow \infty}\left|\frac{1}{H^4} \sum_{h,l \in[H]^2}\left(\prod_{v \in [2]} T^{h \cdot v+l \cdot v} f_1\right) \otimes \cdots \otimes\left(\prod_{v \in [2]} T^{d(h \cdot v)+l \cdot v} f_d\right)\right|^2 d \mu^{(d)}(\mathbf{x})\right)^{\frac{1}{2}} \\
		&\leq  \left(\int_{X^d} \limsup _{H \rightarrow \infty} \left\lvert\, \frac{1}{H^4} \sum_{h_1, h_2=0}^{H-1}\left(T^{ h_1+l_1} f_1 T^{ h_2+l_2} f_1 T^{h_1+l_1+h_2+l_2} f_1\right) \otimes \cdots \otimes\right.\right. \\
		& \ \ \ \ \ \ \ \ \ \ \ \ \ \ \ \ \ \ \ \ \ \ \ \ \  \ \ \ \ \ \ \ \ \ \ \ \ \  \left.\left.\left(T^{d h_1+l_1} f_d T^{d h_2+l_2} f_d T^{d\left(h_1+h_2\right)+(l_1+l_2)} f_d\right)\right|^2 d \mu^{(d)}(\mathbf{x})\right)^{\frac{1}{2}} \\
		&\leq  \left(4 \int_{X^d} \limsup _{H \rightarrow \infty} \sup _{t,s}\left| \sum_{m,n=0}^{H-1} e^{2 \pi i (m t+ns)}\left(T^{ m+n} f_1\right) \otimes \cdots \otimes\left(T^{ m+dn} f_d\right)\right|^2 d \mu^{(d)}(\mathbf{x})\right)^{\frac{1}{2}} \\
		&\lesssim_d \min_{1 \leq i \leq d}\interleave f_i\interleave_{d+3} .
	\end{align*}
\end{proof}

\subsection{Proof of Theorem \ref{thm-A}}\label{3.2}

\begin{thm}\cite{HSY19}\label{Nd-s.e.}
	Let $(X,T)$ be a uniquely ergodic t.d.s. with invariant measure $\mu$ and $d\geq 2$. Assume that the measure theoretic factor map $\pi_{d-2}:X\rightarrow Z_{d-2}$ is (equal $\mu$-a.e. to) a continuous factor map. Then $(N_d(X),\G_d)$ is uniquely ergodic, and the uniquely invariant measure is the Furstenberg self-joining $\mu^{(d)}$.
\end{thm}

\begin{thm}\label{main1}
	Let $d\geq 1$. If $(X,T)$ is a CF-Nil($\infty$) system, then so is $({ N}_{d}(X,T),\G _d)$.
\end{thm}
\begin{proof}
	It is clear for $d=1$. We assume that $d\geq 2$. Suppose $\mu$ is the unique invariant measure of $(X,T)$. By \cite[Theorem 3.11]{SY}, $(N_{d}(X,T),\G _d)$ is minimal, and it follows from Theorem \ref{Nd-s.e.} that $(N_{d}(X,T),\G_d)$ is uniquely ergodic and the unique measure is $\lambda:=\mu^{(d)}$.
	
	Let us denote by $\omega_\infty$ the uniquely invariant measure of $(W_\infty(X),T)$.
	We prove that $\W_\infty(N_d(X))=\Z_\infty(N_d(X))$. 
	By the definition, we know that $\W_\infty(N_d(X))\subset \Z_\infty(N_d(X))$.
	So it is enough to show that $ \Z_\infty(N_d)\subset \W_\infty(N_d)$.
	
	Let $F=\bigotimes_{1\leq j\leq d }f_j\in L^{\infty}(\lambda)$ and suppose $\E (F|\W_\infty(N_d(X))\equiv0$. 
	By Theorem \ref{pronil N} we have $\E(C^{|v|}F|N_d(W_\infty))\equiv0$ for every $v\in[k+1]$, where $C^{|v|}F$ denoted by $\bigotimes_{1\leq j\leq d}C^{|v|}f_j$. 
	For every fixed $k\in\N$,
	\begin{equation*}
		\begin{aligned}
			\interleave F\interleave_{k,\mu^{(d)}}^{2^{k}}=&\int \prod_{v\in[k]}C^{|v|}Fd(\mu^{(d)})^{[k+1]}\\
			=&\lim_{N\ra\infty}\frac{1}{N^{2k}}\sum_{h,l\in[0,N-1]^k}\int\prod_{v\in[k]} C^{|v|}F\left(\sigma_d^{v\cdot l} \tau_d^{v\cdot h} \x\right)d\mu^{(d)}
			\\
			=&
			\lim _{N \rightarrow \infty} \frac{1}{N^{2k}} \sum_{h,l \in[0,N-1]^{k}} \int_{X^d}\bigotimes_{j=1}^{d}\left(\prod_{v \in [k]} T^{j(h \cdot v)+l \cdot v} f_j\right) d \mu^{(d)}(\mathbf{x})\\
			\lesssim& _{d}\min _{1 \leq j \leq d}\interleave f_j\interleave_{d+k+1}\lesssim_d \min_{1\leq j\leq d}\interleave f_j\interleave_{\infty}
		\end{aligned}
	\end{equation*}
	by Theorem \ref{main dom} and the definition of Host-Kra seminorms. For $1\leq j\leq d$, let $f_j=\E(f_j|\Z_\infty)+[f_j-\E(f_j|\Z_\infty)]$, we have  $$\interleave F\interleave_{k+1,\mu^{(d)}}^{2^{k+1}}=\interleave \bigotimes_{j=1}^d\E(f_j|\Z_\infty)\interleave_{k+1,\mu^{(d)}}^{2^{k+1}}=\interleave \bigotimes_{j=1}^d\E(f_j|\W_\infty)\interleave_{k+1,\mu^{(d)}}^{2^{k+1}}$$
	since $\interleave f_j-\E(f_j|\Z_\infty) \interleave_\infty=0$ and $(X,T)$ is CF-Nil($\infty$). Hence, 
	\begin{equation*}
		\begin{aligned}
			&\interleave F\interleave_{k+1,\mu^{(d)}}^{2^{k+1}}\\
			=&\lim_{N\ra\infty}\frac{1}{N^{2k+1}}\sum_{h,l\in[0,N-1]^k}\sum_{n\in[0,N-1]}\int \prod_{v\in[k+1]}C^{|v|}\bigotimes_{j=1}^{d}\E (f_{j}|\W_\infty)(T^{v\cdot l+(v\cdot h)j+nj} x)d \omega_\infty\\
			=&\lim_{N\ra\infty}\frac{1}{N^{2k}}\sum_{h,l\in[0,N-1]^k}\int\prod_{v\in[k+1]}C^{|v|}
			\E(F|N_d(\W_\infty))(\sigma_d^{v\cdot l} \tau_d^{v\cdot h}\x)d\omega_\infty^{(d)}=0
		\end{aligned}
	\end{equation*}
	since $\omega_\infty^{(d)}$ is supported on $N_d(\W_\infty)$ and the assumption $\E (F|\W_\infty(N_d(X))\equiv0$. 
	
	This implies that $\E (F|\Z_k({N_d(X)}))=0$ for every $k\in\N$. Thus, 	$\Z_\infty(N_d(X))\subset \W_\infty(N_d(X))$ follows from the fact $\E (F|\Z_\infty({N_d(X)}))=\E (F|\bigvee_{k\in\N}\Z_k({N_d(X)}))=0$.
	
\end{proof}

Combining Theorems \ref{CFNilinfty model} and \ref{main1}, we obtain Theorem \ref{thm-A}.

\medskip

\subsection{Proof of Theorem \ref{thm-B}}\label{3.3}

We are now ready to prove Theorem \ref{thm-B} which we restate here for convenience.
\begin{thm}\label{thm B}
	Let $(X,\X,\mu,T)$ be an ergodic m.p.s. and $k,d\geq 1$. Then for any ($\infty$-step) nilsequence $\{\psi({m},n)\}_{{(m,n)}\in\ZZ^2}$, any function $f_{j}\in L^{\infty}(\mu), 1\leq j\leq d$, the average
	\begin{equation*}\label{psi f}
		\dfrac{1}{N^{2}}\sum_{m,n=0}^{N-1} \psi(m,n)\prod_{j=1}^{d}f_{j}(T^{n+mj}x)
	\end{equation*}
	converges $\mu$ almost everywhere as $N\ra \infty$.
\end{thm}
\begin{proof}
	Suppose that $\{\psi(m,n):m,n\in\ZZ\}$ is a nilsequence. The proof can be divided into two steps.
	
	\medskip
	\noindent{\bf Step 1:} 
	There exists a strictly ergodic model $(\hat{X},\hat{T})$ of $(X,\X,\mu,T)$ such that $(N_d(\hat{X},\hat{T}),\G_d)$ is CF-Nil($\infty$)  by Theorem \ref{thm-A}. So we may assume that $(X,T,\mu)$ is strictly ergodic with $(N_d(X,T),\G_d)$ being CF-Nil($\infty$).
	
	It follows from Theorem \ref{main1} that the induced system $(N_d(X,T),\G_d)$ is CF-Nil($\infty$) for any $d\in\N$.
	Since Theorem \ref{219} still holds for $\ZZ^2$-actions, for any continuous function $F\in \C(N_d(X,T))$ and $x^{(d)}=(x,x,\dots,x)\in N_d(X,T)$, the limit $$\lim_{N\ra\infty}\dfrac{1}{N^2}\sum_{m,n=0}^{N-1}\psi(m,n)F\left((T\times \dots\times T)^m (T\times T^2\times \dots\times T^d)^{n}(x,x,\dots,x)\right)$$
	exists.
	In particular, let $F(\x)=\bigotimes_{j=1}^{d} f_j(x_j)$ where $f_j\in \C(X)$.
	Then the limit
	\begin{equation*}
		\lim_{N\ra\infty}\dfrac{1}{N^{2}}\sum_{m,n=0}^{N-1} \psi(m,n)\prod_{j=1}^{d}f_{j}(T^{m+nj}x)
	\end{equation*}
	exists.
	
	\medskip
	\noindent{\bf Step 2:} 
	Fix $f_1,\dots,f_{d-1}\in L^{\infty}(\mu)$ and let $k>0$.
	Without loss of generality, we may assume $\|\psi_j\|_{\infty}\leq 1$ and $\|f_j\|_{\infty}\leq 1$ for all $1\leq j\leq d$.
	Choose continuous functions $g_{j,k}$ such that 
	$\|g_{j,k}\|_{\infty}\leq 1$ and $\|f_{j}-g_{j,k}\|_{L^2(\mu)}\leq 1/kd$  for all $1\leq j\leq d$. We have
	\begin{equation*}\label{petZd}
		\begin{aligned}
			&\dfrac{1}{N^{2}}\sum_{m,n=0}^{N-1} \psi(m,n)\prod_{j=1 }^{d}f_{j}(T^{m+jn}x)
			\\
			= &	\dfrac{1}{N^{2}}\sum_{m,n=0}^{N-1} \psi(m,n)\left[\prod_{j=1}^{d}g_{j,k}(T^{m+jn}x)+\left(\prod_{j=1}^{d}f_{j}(T^{m+jn}x)-\prod_{j=1}^{d}g_{j,k}(T^{m+jn}x)\right)\right].
		\end{aligned}
	\end{equation*}
	
	By Pointwise Ergodic Theorem for $\ZZ^2$ actions applying to $(m,n)\mapsto T^{m+(j-1)n}$ we have that for all $1\leq j\leq d$, there exists $X_0\subset X$ with $\mu(X_0)=1$ such that for any $x\in X_0$,
	\begin{equation*}
		\dfrac{1}{N^2}\sum_{m,n=0}^{N-1}\left|f_j(T^{m+(j-1)n}x)-g_{j,k}(T^{m+(j-1)n}x)\right|^2\ra\|f_j-g_{j,k}\|_{2}^{2}, N\ra\infty.
	\end{equation*}
	Hence by \cite[Lemma 4.6]{HSY19} we have
	$$	
	\begin{aligned}
		&\limsup_{N\ra\infty}
		\left|\dfrac{1}{N^{2}}\sum_{m,n=0}^{N-1} \psi(m,n)\prod_{j=1}^{d}f_{j}(T^{m+jn}x)-\dfrac{1}{N^{2}}\sum_{m,n=0}^{N-1} \psi(m,n)\prod_{j=1}^{d}g_{j,k}(T^{m+jn}x)\right|
		\\
		&\leq \sum_{j=1}^{d}\left[\limsup_{N\ra\infty}\dfrac{1}{N^2}\sum_{m,n=0}^{N-1}|\psi(m,n)|\cdot\left|f(T^{m+jn}x)-g_{j,k}(T^{m+jn}x)\right|\right]\\
		&\leq \sum_{j=1}^{d}\left[\limsup_{N\ra\infty}\left(\dfrac{1}{N^2}\sum_{m,n=0}^{N-1}\left|f(T^{m+jn}x)-g_{j,k}(T^{m+jn}x)\right|^2 \right)^{1/2}\cdot\left(\dfrac{1}{N^2}\sum_{m,n=0}^{N-1}|\psi(m,n)|^2\right)^{1/2}\right]\\
		&\leq\sum_{j=1}^{d}{\|f_j-g_{j,k}\|_{L^2(\mu)}}\cdot {\|\psi\|_{L^2(\mu)}}\leq 1/k.
	\end{aligned}
	$$
	
	Then by Step 1 we have 
	\begin{equation*}\label{3.8}
		\lim_{N\ra\infty}	\dfrac{1}{N^2}\sum_{m,n=0}^{N-1}\psi(m,n)\prod_{j=1}^d g_{j,k}(T^{m+nj}x)=c_k
	\end{equation*} exists for any fixed $x\in X_0$. And 	
	$$	
	\begin{aligned}
		&\lim_{N\ra\infty}
		\left|\dfrac{1}{N^{2}}\sum_{m,n=0}^{N-1} \psi(m,n)\prod_{j=1 }^{d}g_{j,k}(T^{m+jn}x)-\dfrac{1}{N^{2}}\sum_{m,n=0}^{N-1} \psi(m,n)\prod_{j=1 }^{d}g_{j,k'}(T^{m+jn}x)\right|
		\\
		&\leq \sum_{j=1}^{d}\left[\lim_{N\ra\infty}\left(\dfrac{1}{N^2}\sum_{m,n=0}^{N-1}\left|g_{j,k}(T^{m+jn}x)-g_{j,k'}(T^{m+jn}x)\right|^2 \right)^{1/2}\cdot\left(\dfrac{1}{N^2}\sum_{m,n=0}^{N-1}|\psi(m,n)|^2\right)^{1/2}\right]\\
		&\leq\sum_{j=1}^{d}{\|g_{j,k}-g_{j,k'}\|_{L^2(\mu)}}\cdot {\|\psi\|_{L^2(\mu)}}\leq 2/k
	\end{aligned}
	$$ for any $k'\geq k$. Therefore $\{c_k\}_{k=1}^{\infty}$ is Cauchy sequence and there exists a constant $c$ such that $\lim_{}c_k=c$ and $\left|c_k-c\right|\leq 2/k$.
	Thus 
	\begin{equation*}
		\limsup_{N\ra\infty}
		\left|\dfrac{1}{N^{2}}\sum_{m,n=0}^{N-1} \psi(m,n)\prod_{j=1 }^{d}f_{j}(T^{m+jn}x)-c\right|\leq 3/k.
	\end{equation*}
	Since $k$ is arbitrary, the proof is completed.
\end{proof}

\bigskip

\section{The proof of Theorem \ref{thm-C}}\label{s4}
In what follows, we assume all polynomials $p_1,p_2,\dots,p_l\in\ZZ[m,n]$ have zero constant terms. 
The set of polynomials $\{p_1,\dots,p_l\}$ is 
\begin{enumerate}
	\item {\bf pairwise independent}, if $p_i/p_j$ is non-constant for any $1\leq i\ne j\leq l$.
	
	\item {\bf totally pairwise independent}, if $p_i^{(t_1,s_1)}/p_j^{(t_2,s_2)}$ is non-constant for all $1\leq i,j\leq l,(t_1,s_1),(t_2,s_2)\in\ZZ^2$ where $p^{(t,s)}(m,n):=p(m+t,n+s)$.
\end{enumerate}
It is clear that if $p_1,\dots,p_l$ have different degrees, they are totally pairwise independent.

\medskip
And in this subsection, we are going to prove the next theorem, which is a stronger form of Theorem \ref{thm-C}.
\begin{thm}\label{main c}
	Let $T,S_1,S_2,\dots,S_l$ be measure preserving transformations acting on  a probability space $(X,\X,\mu)$ such that $S_1,S_2,\dots,S_l$
	commute. Let $d\geq 2$ and $p_1,\cdots,p_l\in\mathbb{Z}[m,n]$ be non-linear totally pairwise independent integer polynomials. Then for any $f_1,\cdots,f_d, g_1,\dots,g_l\in L^{\infty}(\mu)$, the limit
	\begin{equation}\label{L2 limit}
		\lim _{N \to \infty} \frac{1}{N^2} \sum_{m,n=1}^{N}  \prod_{j=1}^{d} f_j\left(T^{m+j n}x\right) \cdot\prod_{i=1}^{l} g_i\left(S_i^{p_i(m,n)}x\right)
	\end{equation}
	exists in $L^2(\mu)$.\end{thm}

\medskip

\subsection{Furstenberg systems of sequences}\label{4.1}
\begin{de}
	Let $
	\left([1,N_k]^s\right)_{k \in \mathbb{N}}$
	be a sequence of cubes with $N_k \rightarrow \infty$ and $I$ be a bounded closed interval in $\R$. We say that the sequence $z=(z(n))_{n \in \mathbb{Z}^s}$ with values in I admits correlations on $([1,N_k]^s)$, 
	if the limit
	$$
	\lim _{k \rightarrow \infty} \dfrac{1}{N_k^s} \sum_{n\in [1,N_k]^s} \prod_{j=1}^l z\left(n+n_j\right)
	$$
	exists for all $l \in \mathbb{N}$ and all $n_1,\dots,n_l \in \ZZ^s$.
\end{de}

Let $\Omega:=I^{\ZZ^s}$. The elements of $\Omega$ are written $\omega:=(\omega(n))_{n \in \mathbb{Z}^s}$ and the shifts $\sigma_j: \Omega \rightarrow \Omega,1\leq j\leq s$, are defined by $(\sigma_j \omega)(n):=\omega(n+\ep_j)$ where $\ep_j=(0,\dots,1,\dots,0)$ with $1$ the $j$-th component and $0$ in the others. We consider the sequence $z$ as an element of $\Omega$. We conclude that if the sequence $z: \mathbb{Z}^s \rightarrow I$ admits correlations on $([1,N_k]^s)$, then for all $f \in \C(\Omega)$ the following limit exists
$$
\lim _{k \rightarrow \infty} \frac{1}{N_k^s} \sum_{n\in[1,N_k]^s} f\left(\sigma_1^{n^1}\dots\sigma_s^{n^s} z\right).
$$
The following weak*-limit exists
$$
\nu:=\lim _{k \ra \infty} \frac{1}{N_k^s} \sum_{n\in[1,N_k]^s} \delta_{\sigma_1^{n^1}\dots\sigma_s^{n^s} z}.
$$
and we say that the point $z$ is generic for $\nu$ along $\left(\left[1,N_k\right]^s\right)$. 
\begin{de}
	Let $I$ be a compact interval and let $z: \ZZ^s \rightarrow I$ be a sequence that admits correlations on $
	\left([1,N_k]^s\right)$, and $\Omega, \sigma_j$, and $\nu$ as above.
	We call $(\Omega, \nu, \langle\sigma_1,\dots,\sigma_s\rangle)$ the {\bf Furstenberg system associated with $z$ on $([1,N_k]^s)$}. 

We let $F_{{\bf 0}} \in \mathcal{C}(\Omega)$ be defined by $F_{{\bf 0}}(\omega):=\omega({{\bf 0}}), \omega \in \Omega$, and call it the ${\bf 0}$-th coordinate projection. Note that $F_{{\bf 0}}\left(\sigma_1^{n^1}\dots \sigma_s^{n^s} z\right)=z(n)$ for every $n=(n^1,\dots,n^s) \in \mathbb{Z}^s$ and
\begin{equation}\label{corresp}
	\lim _{k \rightarrow \infty} \frac{1}{N_k^s} \sum_{n\in[1,N_k]^s} \prod_{j=1}^l z\left(n+n_j\right)=\int \prod_{j=1}^l \sigma_1^{n_j^1}\dots \sigma_s^{n_j^s} F_{{\bf 0}} d\nu
\end{equation}
for all $l \in \mathbb{N}, n_1, \ldots, n_l \in \ZZ^s$. This identity is referred to as the {\bf Furstenberg correspondence principle}.

We say that the sequence $z$ has a unique Furstenberg system if $z$ is generic for a measure along $([1,N_k]^s)_{N\in\N}$.
\end{de}


\begin{prop}\cite{FH}\label{Fur sys}
Let $(X,\mu,T)$ be a system and suppose that $\mu=\int \mu_x d\mu(x)$ is the ergodic decomposition of $\mu$. Then for every $f\in L^{\infty}(\mu)$ and for almost every $x\in X$, the sequence $(f(T^n x))$ has a unique Furstenberg system that is ergodic and a factor of the system $(X,\mu_x,T)$.
\end{prop}

\begin{prop}\label{Fur sys Nd}
Let $(X, \mu, T)$ be a system and $d\geq 2$. Then for $f_j \in L^{\infty}(\mu),1\leq j\leq d$ 
and for almost every $x \in X$, the sequence $\left(\prod_{j=1}^d f_j\left(T^{m+jn} x\right)\right)$ has a unique Furstenberg system that is ergodic with zero entropy.
\end{prop}
\begin{proof}
By Proposition \ref{Fur sys}, there exists $X_0\subset X$ with $\mu(X_0)=1$ such that for every $f_j\in L^{\infty}(\mu),1\leq j\leq d$ and every $x\in X_0$, the sequence $\left(f_j\left(T^{n} x\right)\right):=z_j$ has a unique Furstenberg system,  denoted by $(\Omega_j=I^{\ZZ},\nu_j,\sigma)$.
Let $x\in X_0$. Denote $\pi_j$ the (measurable) factor map from $(X,\mu_x,T)$ to $(\Omega_j=I^{\ZZ},\nu_j,\sigma)$ where $\nu_j=\lim_{N\ra\infty}\frac{1}{N}\sum_{n\in[1,N]}\delta_{\sigma^n z_j}$. We have $(\pi_i)_*\mu=\mu\circ(\pi_j)^{-1}=v_j,\forall 1\leq j\leq d$.
Consider the induced system $(\Omega^d,\lambda_x,\G_d(\sigma))$,
where
$$\lambda_x=\lim_{N\ra\infty}\dfrac{1}{N^2}\sum_{m,n=1}^{N}(\sigma\times\sigma^2\dots\times\sigma^{d})^n(\sigma\times\sigma\dots\times\sigma)^m\delta_{z_1}\times\delta_{z_2}\dots\times\delta_{z_d}.$$
Points of $\Omega^d$ are written as $\x=(x_1,\dots,x_d)$, $x_i=(x_i(n))_{n\in\ZZ}$. 

First, we claim  $(\Omega^d,\lambda_x,\G_d(\sigma))$ is the unique Furstenberg system of the sequence $\left(\prod_{j=1}^d f_j\left(T^{m+jn} x\right)\right)_{m,n}:=(a(m,n))_{m,n}$. Let $F_{{\bf 0}}(\x):=\bigotimes_{j=1}^{d}f_j(x_j(0))\in \C(\Omega^d)$. Then for any $s\in\N$ and $m_j,n_j\in\ZZ$, $1\leq i\leq s$, we have
$$
\lim _{k \rightarrow \infty} \dfrac{1}{N^2} \sum_{m,n=1}^{N} \prod_{j=1}^s a(m+m_j,n+n_j)=\int \prod_{j=1}^s \left(\tau_d(\sigma)\right)^{n_j} (\sigma^{(d)})^{m_j} F_{{\bf 0}}d\lambda_x.
$$
It follows from Furstenberg correspondence principle that  $(\Omega^d,\lambda_x,\G_d(\sigma))$ where $\G_d(\sigma):=\langle \sigma^{(d)}, \tau_d(\sigma)
\rangle$ is the unique Furstenberg system associated with the sequence $\left(\prod_{j=1}^d f_j\left(T^{m+jn} x\right)\right)_{m,n}$.

It is easy to check $(\pi_1\times\dots\times\pi_d)_*\mu^{(d)}=\mu^{(d)}\circ(\pi_1\times\dots\times\pi_d)^{-1}=\lambda_x$. So the map $\pi_1\times\dots\times\pi_d: (X^d,\mu_x^{(d)},\G_d(T))\ra (\Omega^d,\lambda_x,\G_d(\sigma))$ is naturally a factor map. 
By model theory from Weiss \cite{W}, there exists uniquely ergodic model $(\hat{X},\hat{\mu_x},\hat{T})$ such that $(X^d,\mu_x^{(d)},\G_d(T))$ and $(\hat{X}^d,\hat{\mu}_x^{(d)},\G_d(\hat{T}))$ are measure theoretically isomorphic, and hence, they have the same entropy. It is clear that $(\hat{X}^d,\hat{\mu}_x^{(d)},\G_d(\hat{T}))$ and $(N_d(\hat{X}),\hat{\mu}_x^{(d)},\G_d(\hat{T}))$ are measurably isomorphic since $N_d(\hat{X})$ is the support of $\hat{\mu}_x^{(d)}$.

At last, we claim that the entropy of the system $(\Omega^d,\lambda_x,\G_d(\sigma))$ is zero: 	
It follows from \cite[Theorem 1.1]{WZ} that $(N_d(\hat{X}),\G(\hat{T}))$ has zero entropy when $d\geq 2$. The variational principle implies  $(N_d(\hat{X}),\hat{\mu}_x^{(d)},\G(\hat{T}))$ has zero entropy, then so have $(X^d,\mu_x^{(d)},\G_d(T))$ and $(\Omega^d,\lambda_x,\G_d(\sigma))$.
\end{proof}

\begin{prop}\label{67}
Let $p_1,\dots,p_l\in\ZZ^2[m,n]$
be totally pairwise independent polynomials, and $S_1,\dots,S_l$ be commuting measure preserving transformations acting on $(X,\X, \mu)$, and $g_1,\dots,g_l \in$ $L^{\infty}(\mu)$.
If $\mathbb{E}\left(g_j  |  \mathcal{Z}_{\infty}(S_j)\right)=0$ for some $j$, then
$$\E(G_{{\bf 0}}|\Pi\left(\Omega, \nu_x, \langle\sigma_1,\sigma_2\rangle\right))=0$$
for almost every $x\in X$,
where $\left(\Omega, \nu_x, \langle\sigma_1,\sigma_2\rangle\right)$ is the Furstenberg system  associated with the sequence $\left(\prod_{j=1}^{l}g_j(S_j^{p(m,n)}x)\right)$ and $G_{{\bf 0}}$ is the ${\bf 0}$-th coordinate projection.
\end{prop}
\begin{proof}
Let  $p_1,\dots,p_l:\ZZ^2\ra\ZZ$ be totally pairwise independent on every variables integral polynomials.
By extending \cite[Theorem 2.8]{NB} to the case $p_1,\dots,p_l:\ZZ^2\ra\ZZ$ (see \cite[page 14]{NB}), we obtain immediately that the limit
\begin{equation}\label{prog g_j fur}
	\lim_{N\ra\infty}\dfrac{1}{N^2}\sum_{m,n=0}^{N-1}\left(\prod_{j=1}^{l}S_j^{p_j(m,n)}g_j\right)\cdot \prod_{i=1}^{t}\left(\prod_{j=1}^{l}S_j^{p_j(m+m_i,n+n_i)}g_j\right)	
\end{equation}
is $0$ in $L^2$ for any $t\in\N,(m_1,n_1),\dots,(m_t,n_t)\in\ZZ^2$ (not necessary distinct) when $\E(g_j|\Z_{\infty}(S_j))=0$ for some $j$, where $g_1,\dots,g_l\in L^{\infty}(\mu)$.	

Since mean convergence of sequences of functions implies pointwise almost everywhere convergence along a subsequence \cite[Section 2.4]{Folland}, there exists a subsequence of $[1,N_k]^2$ and a full measure subset $X_0$ such that for $x\in X_0$ all the limits (\ref{prog g_j fur}) exists.

For fixed $x\in X_0$, let $(\Omega, \nu_x, \langle\sigma_1,\sigma_2\rangle)$ be the Furstenberg system associated with the sequence $\left(\prod_{j=1}^{l}g_j(S_j^{p(m,n)}x)\right)$ on $[N_k]^2$ and let $G_{{\bf 0}}$ be the ${\bf 0}$-th coordinate projection.
Then using the correspondence principle (\ref{corresp}) we can get $$\int G_{{\bf 0}}\cdot \prod_{i=1}^t  \sigma_1^{m_i} \sigma_2^{n_i} G_{{\bf 0}} d\nu_x=0$$
for all $t\geq 0$ and $(m_1,n_1),\dots,(m_t,n_t)\in\N^2$.
By Lemma \ref{pinsker}, for all $x\in X_0$, $\E(G_{{\bf 0}}|\Pi\left(\Omega, \nu_x, \langle\sigma_1,\sigma_2\rangle\right))=0$.
\end{proof}
\subsection{Proof of Theorem \ref{main c}}\label{4.2}

The following proposition is inspired by  \cite[Proposition 4.1]{FH}. 
\begin{prop}\label{convergL2=0}
Let $T, S_1, \dots, S_l$ be invertible measure preserving transformations acting on a Lebesgue space $(X, \mathcal{B}, \mu)$ with $S_1, \dots, S_l$ be commuting. Let $d\geq 2$,  $f_1,\dots,f_d, g_1, \dots, g_l \in L^{\infty}(\mu)$ and  $p_1, \dots, p_l\in \ZZ[m,n]$ be
non-linear and totally pairwise independent. If there exists $j\in\{1,2,\dots,l\}$ such that $\E(g_j|\Z_{\infty}(S_j))=0$, then the limit
$$
\lim _{N \rightarrow \infty} \frac{1}{N^2} \sum_{m,n=1}^N \prod_{j=1}^d T^{m+j n}f_j\cdot \prod_{j=1}^l S_j^{p_j(m,n)}g_j
=0$$
in $L^2(\mu)$.
\end{prop}
\begin{proof}
Suppose that the conclusion fails. Then there exist $\varepsilon>0$ and $N_k \rightarrow \infty$ such that
\begin{equation}\label{6.1}
	\left\|\frac{1}{N_k^2} \sum_{m,n=1}^{N_k}  \prod_{j=1}^d T^{m+jn} f_j\cdot \prod_{j=1}^l S_j^{p_j(m,n)}g_j \right\|_{L^2(\mu)} \geqslant \varepsilon
\end{equation}
for every $k \in \mathbb{N}$. By passing to a subsequence if necessary, we may suppose that for $\mu$-almost every $x \in X$ the sequence $\left(g\left(S^{p(m,n)} x\right)\right)$ admits correlations on $
\left(\left[N_{x,k}\right]^2\right)$. Let $\left(J^{\ZZ\times\ZZ}, \nu_x, \langle \sigma_1,\sigma_2\rangle\right)$ be the corresponding Furstenberg system. Moreover, since $\mathbb{E}\left(g_j  |  \mathcal{Z}_{\infty}(S_j)\right)=0$ for some $j$, by 
Proposition \ref{67} we have that for $\mu$-almost every $x \in X$ the {\bf 0}-th coordinate projection $G_{{\bf 0}}: J^{\ZZ\times\ZZ} \rightarrow \mathbb{R}$ satisfies $\mathbb{E}_{v_x}\left(G_{{\bf 0}}  |  \Pi\left(J^{\ZZ\times\ZZ}, v_x, \langle\sigma_1,\sigma_2\rangle\right)\right)=0$.

Then by Proposition \ref{Fur sys Nd}, for $\mu$-almost every $x\in X$ the sequence $\left(\prod_{j=1}^d f_j\left(T^{m+jn} x\right)\right)$ admits correlations on the sequence of intervals $([N]^2)$ and the corresponding Furstenberg system $(I^{\ZZ\times\ZZ},\lambda_x,\langle\sigma_1,\sigma_2\rangle)$ is ergodic and has zero entropy.\footnote{It is isomorphic to the system $(N_d(\Omega),\lambda_x,\G_d(\sigma))$ defined in the proof of Proposition \ref{Fur sys Nd}.}

We write $F_{{\bf 0}}: I^{\ZZ\times\ZZ} \rightarrow \mathbb{R}$ for the {\bf 0}-th coordinate projection.
Let $X_0$ be a subset of $X$ with $\mu\left(X_0\right)=1$ and such that the previous properties hold for all $x \in X_0$. We claim that
\begin{equation}\label{6.2}
	\lim _{k \rightarrow \infty} \frac{1}{N_{x, k}^{2}} \sum_{m,n=1}^{N_{x, k}} \prod_{j=1}^d f_j\left(T^{m+j n} x\right) \cdot \prod_{j=1}^l g_j\left(S_j^{p_j(m,n)} x\right) =0 \text{ for $\mu$-a.e. } x\in X_0.
\end{equation}
If this is shown, then we get a contradiction from (\ref{6.2}) using the bounded convergence theorem.

Now we suppose that (\ref{6.2}) fails. Then there exists a subset $X_1$ of $X_0$ with $\mu\left(X_1\right)>0$ such that for every $x \in X_1$ there exists a subsequence $\left(N_{x, k}^{\prime}\right)$ of $\left(N_{k,x}\right)$ such that
\begin{equation}\label{6.3}
	\lim _{k \rightarrow \infty} \frac{1}{N_{x, k}^{\prime 2}} \sum_{m,n=1}^{N_{x, k}^{\prime}} \prod_{j=1}^d f_j\left(T^{m+j n} x\right) \cdot \prod_{j=1}^l g_j\left(S_j^{p_j(m,n)} x\right) \text { exists and is non-zero.}
\end{equation}

Let $x \in X_1$ be fixed for the moment. Let $w, z \in I^{\mathbb{Z}\times\mathbb{Z}}$ be defined by $$w(m,n):=\prod_{j=1}^d f_j\left(T^{m+jn} x\right) \text{ and } z(m,n):=\prod_{j=1}^{l}g_j\left(S^{p(m,n)} x\right),\forall m,n\in\ZZ.$$
The sequence $\left(N_{x, k}^{\prime}\right)$ admits a subsequence $\left(N_{x, k}^{\prime \prime}\right)$ along which the point $(w, z)$ of $I^{\mathbb{Z}\times\mathbb{Z}} \times I^{\mathbb{Z}\times \mathbb{Z}}$ is generic on $\left(I^{\mathbb{Z}\times\mathbb{Z}} \times J^{\mathbb{Z}\times\mathbb{Z}}, \langle\sigma_1,\sigma_2\rangle \right)$ for some measure $\rho_x$ on this space. This means that for
every continuous function $\Phi$ on $I^{\mathbb{Z}\times\ZZ} \times J^{\mathbb{Z}\times\ZZ}$ we have
$$
\lim _{k \rightarrow \infty} \frac{1}{(N_{x, k}'')^{2}}\sum_{m,n=1}^{N_{x, k}''} \Phi\left(\sigma_1^m \sigma_2^n w, \sigma_1^m\sigma_2^n z\right)=\int \Phi d \rho_x.
$$

In particular, since  $F_{{\bf 0}}$ and $G_{{\bf 0}}$ are continuous on $I^{\ZZ\times\ZZ}$ and $J^{\ZZ\times\ZZ}$ respectivity, then by definition, we obtain
\begin{equation}\label{6.4}
	\lim _{k \rightarrow \infty} \frac{1}{N_{x, k}^{\prime \prime 2}}  \sum_{m,n=1}^{N_{x, k}^{\prime}} \prod_{j=1}^d f_j\left(T^{m+j n} x\right) \cdot \prod_{j=1}^l g_j\left(S_j^{p_j(m,n)} x\right) =\int G_{{\bf 0}} \otimes F_{{\bf 0}} d \rho_x.
\end{equation}

Note that the measure $\rho_x$ is invariant under $\langle\sigma_1,\sigma_2\rangle$. Furthermore, since $w$ is generic for the measure $v_x$ on $(I^{\ZZ\times\ZZ}, \langle\sigma_1,\sigma_2\rangle)$ along $([N]^2)$ and $z$ is generic for the measure $\lambda_x$ on $(J^{\ZZ\times\ZZ}, \langle\sigma_1,\sigma_2\rangle)$ along $\left(\left[N_{k,x}\right]\right)$, and $\left([N_{x, k}^{\prime}]^2\right)$ is a subsequence of $\left([N_{x,k}]^2\right)$, the two coordinate projections of $\rho_x$ are $\lambda_x$ and $\nu_x$, respectively, i.e.,  $\rho_x$ is a joining of the systems $\left(I^{\ZZ\times\ZZ}, \lambda_x, \langle\sigma_1,\sigma_2\rangle\right)$ and $\left(J^{\ZZ\times\ZZ}, \nu_x, \langle \sigma_1,\sigma_2\rangle\right)$.

Recall that for $\mu$-almost every $x \in X_1$ we have $\mathbb{E}_{v_x}\left(G_{{\bf 0}}  |  \Pi\left(J^{\ZZ\times\ZZ}, v_x, \langle\sigma_1,\sigma_2\rangle\right)\right)=0$ and the system $\left(I^{\ZZ\times\ZZ}, \lambda_x, \langle\sigma_1,\sigma_2\rangle\right)$ has entropy zero. By Proposition \ref{FH2.1}, we have
$$
\int G_{\bf 0} \otimes F_{\bf 0} d \rho_x=0
$$
for $\mu$-almost every $x \in X_1$. Taken together with identity (\ref{6.4}), this gives a contradiction by (\ref{6.3}). This proves (\ref{6.2}) and finishes the proof.
\end{proof}

\begin{proof}[Proof of Theorem \ref{main c}]
By Proposition \ref{convergL2=0},  the existence of (\ref{L2 limit}) is equal to the existence of
$$\lim _{N\rightarrow\infty} \frac{1}{N^2} \sum_{m,n=1}^N  \prod_{i=1}^d T^{m+j n}f_j\cdot\prod_{j=1}^{l} S_j^{p_j(m,n)}\E(g_j|\Z_{\infty}(S_j)).$$
So we can restrict to the case where $g_j$ is measurable with respect to $\Z_{\infty}(S_j)$ for every $1\leq j\leq l$.
By an $L^2(\mu)$ approximation argument we can assume that $g_j$ is measurable with respect to the factor $\Z_{k_j}(S_j)$ for some $k_j\in\N$, $1\leq j\leq l$.

By \cite[Proposition 3.1]{CFH}, for every $\ep>0$ there exists $\tilde{g_j}\in L^{\infty}(\mu)$ such that:

\noindent(i) $\tilde{g_j}$ is measurable with respect to $\Z_{k_j}(S_j)$ and $\|g_j-\tilde{g_j}\|_{L^1(\mu)}<\ep$;

\noindent(ii) for $\mu$ a.e. $x\in X$ the sequence $(\tilde{g_j}(S_j^{n}x))$ is a $k_j$-step nilsequence.

Therefore, it suffices to show the existence in $L^2(\mu)$ of the limit
\begin{equation}\label{tilde g}
	\lim _{N \rightarrow \infty} \frac{1}{N^2} \sum_{m,n=1}^N \prod_{i=1}^d T^{m+j n}f_j \cdot \prod_{j=1}^l S_j^{p_j(m,n)}\tilde{g_j}.
\end{equation}
By \cite[Theorem B*]{Leib2} for $\mu$ a.e. $x\in X$ the sequence $(\tilde{g_j}(S_j^{p_j(m,n)}x))$ is an $l$-step nilsequence for some $l=l(k_j,p_j)$.
Then $\prod_{j=1}^l \tilde{g_j}(S_j^{p_j(m,n)} x)$ is an $r$-step nilsequence for some $r=r(k_1,\dots,k_l,p_1,\dots,p_l)$.

Let $X_0$ be the full measure subset of $X$ for which this property holds. Let also $X_1$ be the subset of $X$ associated with $f$ by Theorem \ref{thm B}. Then for every $x\in X_0\cap X_1$ the limit (\ref{tilde g}) exists in $L^2(\mu)$.
\end{proof}

\bigskip

\section{Counterexample for  multiple convergence}\label{s5}
Theorem \ref{GLThmC} shows that CF-Nil systems are related to everywhere convergence.
It is natural to ask if $(X,T)$ is a CF-Nil($k$) system, whether the multiple average of $$\lim_{N\rightarrow \infty}\frac{1}{N}\sum_{n=0}^{N-1}f_1(T^n x)f_2(T^{2n} x)$$
exists for any continuous function $f_1,f_2\in \C(X)$ and any $x\in X$. Unfortunately, it is not true. In this section we will prove Theorem \ref{thm-D} by using Floyd-Auslander systems.
\subsection{Floyd-Auslander systems}
Floyd \cite{Floyd} constructed a topological dynamical system with a non-homogeneous space and a
minimal homeomorphism in 1949. Later Auslander introduced a minimal, mean-L-stable, yet not distal t.d.s., which is a modification of Floyd’s example. We call this kind of systems Floyd-Auslander (F-A) systems (see the concrete construction and more details in \cite{Aus}).

\begin{lem}\label{7.4}\cite[Lemma 5.6]{GHSWY20}\label{7.4}
Let $\phi:X\ra Y$ be a proximal (for example, almost 1-1) extension between minimal systems. Then $X_{eq}=Y_{eq}$ where $X_{eq}, Y_{eq}$ are the maximal equicontinuous factors of $X,Y$ respectively.
\end{lem}

\begin{prop}\cite[Proposition  4.3]{Cao}\label{Cao}
Let $(X,T)$ be a Floyd-Auslander system. Then $(X,T)$ is an almost 1-1 extension of an odometer (adding machine) $(Y,S)$.\footnote{For definitions and more information about odometers, refer to \cite{Dow}.}
\end{prop}

It is well-known that the adding machine is equicontinuous. 
Then by Lemma \ref{7.4}, the odometer $(Y,S)$ is the maximal equicontinuous factor of an F-A system $(X,T)$.

\begin{thm}\cite[Theorem 4.4]{Cao}\label{Cao}
Let $(X,T)$ be a Floyd-Auslander system, and let $\pi: X \rightarrow Y$ be the almost 1-1 extension from an odometer $(Y, S)$ to $(X, T)$. Then there exists a residual subset $Y_0$ of $Y$ such that for each $y \in Y_0, \pi^{-1}(y)$ is a singleton which is a multiple recurrent point, and for any $y \in Y \backslash Y_0$, each point in $\pi^{-1}(y)$ is not a 2-multiple-recurrent point.\footnote{A point $x\in X$ is $d$-multiple-recurrent if there exists a sequence $n_i\ra\infty$ with $T^{n_i}x,T^{2n_i}x,\dots, T^{dn_i}x\ra x$ simultaneously; is multiple recurrent if is d-multiple-recurrent for any $d\in \N$.}
\end{thm}

\begin{thm}\label{FA is CF}
Floyd-Auslander systems are CF-Nil($\infty$) systems.
\end{thm}
\begin{proof}
Let $(X,T)$ be a Floyd-Auslander system. Adopt the notions in Proposition \ref{Cao}.
Since any equicontinuous system is distal, $Y=W_k(Y)\cong Z_k(Y)$ as m.p.s. for any $k\geq 0$. So $(Y,S)$ is a CF-Nil($k$) system.

Let $\mu$ be the unique measure of $(X,T)$ (the unique ergodicity of $(X,T)$ follows from \cite[Theorem 8]{Aus}). Obviously $\left(\pi\right)_* \mu=\nu$ where $\nu$ is the Haar measure of $(Y,S)$. 
Since the $S$-invariant set $Y\backslash Y_0$ is countable (see \cite{Cao}, Section 4.3), we have $\nu(Y\backslash Y_0)=0$. Thus, $\mu(\pi^{-1}(Y_0))=1$. It follows that $(X,T,\mu)$ is   measurably isomorphic to  $(Y,S,\nu)$. 
Since $\pi$ is almost 1-1 (on a full measure subset) and $(Y,S)$ is CF-Nil($k$), $(X,T)$ is also CF-Nil($k$).
\end{proof}

\subsection{Proof of Theorem D}

We begin the proof with the following lemmas.
\begin{lem}\cite[Lemma 5.2]{GL}\label{generic}
Let $(X, T)$ be a t.d.s. and $x_0 \in X$. Assume that for all $f \in \C(X)$, there exists a constant $c_f \in \mathbb{R}$, depending on $f$, such that :
$$
\lim _{N \rightarrow \infty} \frac{1}{N} \sum_{n=0}^N f\left(T^n x_0\right)=c_f.
$$
Then $x_0$ is generic for some $\mu \in M_T(X)$.
\end{lem}
The following lemma is obvious.
\begin{lem}\label{gene2}
Let $(X, T)$ be a t.d.s. and $\mu \in M_T(X)$. If a point $x \in X$ is generic for $\mu$, then $\mu$ is supported on $\overline{\mathcal{O}}(x)$.
\end{lem}
\begin{proof}[Proof of Theorem \ref{thm-D}]
Let $(X, T)$ be a Floyd-Auslander system, then it is a  CF-Nil($\infty$) system by Theorem \ref{FA is CF}.	

We prove it by contradiction. Suppose that for any $x\in X$ and any $f_1,f_2\in \C(X)$, the limit $$\lim_{N\rightarrow \infty}\frac{1}{N}\sum_{n=0}^{N-1}f_1(T^n x)f_2(T^{2n} x)$$ exists.
Let $x_0 \in X$ and $x_0$ not be a 2-multiple-recurrent point (the existence of such $x_0$ follows from Theorem \ref{Cao}). Set $\left(W,\tau\right)=\left(X\times X,T\times T^2\right)$. By the hypothesis, applying Lemma \ref{generic} to $(W,\tau)$,
the point $(x_0,x_0)$ is generic for some $\mu\in M_{\tau}(W)$. Thus we have
\begin{equation}\label{F}
	\lim_{N\rightarrow \infty}\frac{1}{N}\sum_{n=0}^{N-1}F\left(\tau^n(x_0,x_0)\right)=\int F d\mu
\end{equation}
for any $F\in \C(W)$.

Since $x_0$  is not 2-multiple-recurrent, there exists $\epsilon_0>0$ such that $B_{\epsilon_0}\left((x_0,x_0)\right)$ consists of only finite points $\tau^n\left((x_0,x_0)\right)$.
Consider the continuous function $F\in \C(W)$ with $$F|_{\overline{B}_{\ep_0/2}\left((x_0,x_0)\right)}=1\ \text{ and }\ F|_{W\backslash B_{\epsilon_0}\left((x_0,x_0)\right)}=0.$$
Then the left hand side of (\ref{F}) (LHS) is $0$.

However, RHS $\geq \mu\left(\overline{B}_{\ep_0/2}\left((x_0,x_0)\right)\right)$. Set $B=\bigcup_{n}\tau^n\overline{B}_{\ep_0/2}\left((x_0,x_0)\right)$. By Lemma \ref{gene2}, $\mu$ is supported on $\overline{\O}((x_0,x_0),\tau)$ which is contained in $B$. 
Then we have  $$1=\mu(B)\leq \sum_{n\in \N}\mu\left(\tau^n\overline{B}_{\ep_0/2}\left((x_0,x_0)\right)\right)=\sum_{n\in\N}\mu\left(\overline{B}_{\ep_0/2}\left((x_0,x_0)\right)\right).$$
Thus RHS $\geq  \mu\left(\overline{B}_{\ep_0/2}\left((x_0,x_0)\right)\right)> 0$ which is a contradiction.
\end{proof}


\bigskip

\begin{thebibliography}{SS}
\bibitem{Assani}  I. Assani, {\it Pointwise convergence of ergodic averages along cubes}, J. Anal. Math. {\bf 110} (2010), 241-269.

\bibitem{Aus} J. Auslander, {\it  Mean-L-stable systems}, Ill. J. Math. {\bf 3} (1959), 566–579.

\bibitem{BTZ} V. Bergelson, T. Tao and T. Ziegler.{\it An inverse theorem for the uniformity seminorms associated with the action of $\mathbb{F}^{\infty}_p$}. Geom. Funct. Anal. {\bf 19} (6) (2010), 1539–1596.

\bibitem{CS} P. Candela and B. Szegedy, {\it  Nilspace Factors for General Uniformity Seminorms, Cubic Exchangeability and Limits},  Mem. Amer. Math. Soc. {\bf 287} (2023), no.1425, v+101 pp.

\bibitem{Cao} Y. Cao, {\it  Floyd-Auslander system and its multiple recurrent properties},
(English summary) Topology Appl. {\bf 330} (2023), no. 108475, 11 pp.

\bibitem{JPC} J. P. Conze, {\it Entropie d’un groupe abélien de transformations}, Z. Wahrscheinlichkeitstheor. Verw. 
Geb. {\bf 25} (1972) 11–30.

\bibitem{CF} Q. Chu and N. Frantzikinakis, {\it Pointwise convergence for cubic and polynomial multiple ergodic averages of non-commuting transformations}, Ergodic Theory Dynam. Systems   {\bf 32} (2012), no. 3, 877-897.


\bibitem{CFH} Q. Chu, N. Frantzikinakis and B. Host, {\it Ergodic averages of commuting transformations with distinct degree polynomial iterates}, Proc. Lond. Math. Soc. (3) {\bf 102} (2011), no.5, 801–842.

\bibitem{dlr} T. de la Rue, {\it Notes on Austin’s multiple ergodic theorem}, arXiv:0907.0538.

\bibitem{DDMSY} P. Dong, S. Donoso, A. Maass, S. Shao and X. Ye, {\it Infinite-step nilsystems, independence and complexity}, Ergodic Theory Dynam. Systems, {\bf 33} (2013), 118–143.

\bibitem{Dow}  T. Downarowicz, {\it Survey of odometers and Toeplitz flows}, in: S. Kolyada et al. (eds.), {\it Algebraic and Topological Dynamics}, Contemp. Math., 385, Amer. Math. Soc., Providence, RI, 2005, pp. 7–37.

\bibitem{Floyd} E. E. Floyd, {\it  A nonhomogeneous minimal set}, Bull. Am. Math. Soc. {\bf 55} (1949), 957–960.

\bibitem{NB} N. Frantzikinakis and B. Kuca, {\it Joint ergodicity for commuting transformations and applications  to polynomial sequences}, Invent. Math. {\bf 239} (2025), no. 2, 621–706.

\bibitem{FH} N. Frantzikinakis and B. Host, {\it Multiple recurrence and convergence without commutativity},
J. Lond. Math. Soc. (2) {\bf  107} (2023), no.5, 1635–1659.

\bibitem{Folland} G. B. Folland, {\it Real analysis},
Modern techniques and their applications. 
Pure Appl. Math. (N. Y.) Wiley-Intersci. Publ.
John Wiley $\&$ Sons, Inc., New York, 1999.

\bibitem{vdw} E. Glasner, {\it Ergodic Theory via Joinings},
Mathematical Surveys and Monographs, 101. American Mathematical Society, 2003,  Providence, RI.


\bibitem{Glasner22} E. Glasner, {\it  Topological ergodic decompositions and applications to products of powers of a minimal transformation},
J. Anal. Math.  {\bf 64} (1994), 241-262.

\bibitem{GGY} E. Glasner, Y. Gutman and X. Ye, {\it  Higher order regionally proximal equivalence relations for general group actions},  Adv. Math. {\bf 333} (2018), 1004–1041.


\bibitem{GHSWY20}
E. Glasner, W. Huang, S. Shao, B. Weiss and X. Ye,
{\it Topological characteristic factors and nilsystems}, J. Eur. Math. Soc. {\bf 27} (2025), no. 1, pp. 279–331.

\bibitem{GL} Y. Gutman and Z. Lian, {\it  Maximal pronilfactors and a topological Wiener-Wintner theorem},  Israel J. Math. {\bf 251} (2022), no.2, 495–526.

\bibitem{GL19} Y. Gutman and Z. Lian, {\it  Strictly ergodic distal models and a new approach to the Host-Kra factors}
J. Funct. Anal. {\bf 284}  (2023), no. 4, Paper No. 109779, 54 pp.



\bibitem{HK05} B. Host and B. Kra, {\it Nonconventional ergodic averages and nilmanifolds}, Ann. of Math. (2) {\bf 161} (2005), no.1, 397–488.


\bibitem{HK09} B. Host and B. Kra, {\it  Uniformity seminorms on $l^{\infty}$ and applications}, J. Anal. Math., {\bf 108} (2009), 219–276.

\bibitem{HKbook} B. Host and B. Kra, {\it Nilpotent Structures in Ergodic Theory}, American Mathematical Society, 2018.

\bibitem{HKM} B. Host, B. Kra, and A. Maass, {\it  Nilsequences and a structure theorem for topological dynamical systems}, Adv. Math., {\bf 224} (2010), no.1, 103–129.

\bibitem{HSY19}	W. Huang, S. Shao and X. Ye, {\it  Pointwise convergence of multiple ergodic averages and strictly ergodic models}, J. Anal. Math. {\bf 139} (2019), no.1, 265–305.

\bibitem{HSY-new} W. Huang, S. Shao and X. Ye, \textit{Polynomial Furstenberg joinings and its applications}, Sci. China Math. (2025). https://doi.org/10.1007/s11425-024-2366-2

\bibitem{J} R. I. Jewett. {\it The prevalence of uniquely ergodic systems}, J. Math. Mech.   {\bf19} (1969/70), 717–729.

\bibitem{K} W. Krieger,{ \it On unique ergodicity}, in Proceedings of the Sixth Berkeley Symposium on
Mathematical Statistics and Probability (Univ. California, Berkeley, Calif., 1970/1971),
Vol. II: Probability theory, University of California Press, Berkeley, Calif., 1972, pp. 327–346.

\bibitem{KN} L. Kuipers and H. Niederreiter,  {\it Uniform distribution of sequences,} Pure and Applied Mathematics. New York etc.: John Wiley $\&$ Sons, a Wiley-Interscience Publication. xiv, 390 p. £
13.00 (1974)., 1974.

\bibitem{Leib3} A. Leibman, {\it Convergence of multiple ergodic averages along polynomials of several variables}, Israel J. Math. {\bf 146} (2005), 303–315.

\bibitem{Leib} A. Leibman, {\it Pointwise convergence of ergodic averages for polynomial sequences of translations on a nilmanifold}, Ergodic Theory Dynam. Systems {\bf 25} (2005), no. 1, 201–213.

\bibitem{Leib2} A. Leibman,
{\it Pointwise convergence of ergodic averages for polynomial actions of $\ZZ^d$ by translations on a nilmanifold},
Ergodic Theory Dynam. Systems {\bf 25} (2005), no.1, 215–225.

\bibitem{Leib4} A. Leibman, 
{\it Nilsequences, null-sequences, and multiple correlation sequences,} 
Ergodic Theory Dynam. Systems {\bf 35} (2015), no. 1, 176–191.

\bibitem{LQ}  Z. Lian and J. Qiu, {\it  Pro-nilfactors of the space of arithmetic progressions in topological dynamical systems}, J. Dyn. Differ. Equ. (2022), {\bf 36} (2024), no.3, 2627–2644.



\bibitem{SY} S. Shao and X. Ye,
{\it  Regionally proximal relation of order d is an equivalence one for minimal systems and a combinatorial consequence}, Adv. Math. {\bf 231} (2012), no.3-4, 1786–1817.


\bibitem{W} B. Weiss, {\it Strictly ergodic models for dynamical systems}, Bull. Amer. Math. Soc. (N, S.) {\bf 13} (1985), 143-146.


\bibitem{WZ} Q. Wu and R.  Zhang,
{\it The complexity of an induced system of $\ZZ^2$-actions}, 
J. Differential Equations {\bf 368} (2023), 203–228.

\bibitem{Xiao} R. Xiao. {\it  Multilinear Wiener-Wintner type ergodic averages and its application}, Discrete Contin. Dyn. Syst. {\bf 44} (2024), no.2, 425–446.


\bibitem{Z} T. Ziegler, {\it Universal characteristic factors and Furstenberg averages}, J. Amer. Math. Soc. {\bf 20} (2007), no.1, 53–97.

\end{thebibliography}
\end{document}